\documentclass{amsart}
\usepackage{amstext,amssymb,amsthm,amsopn,newlfont,graphpap,graphics,graphicx,mathrsfs,enumitem}
\allowdisplaybreaks
\usepackage[parfill]{parskip}
\usepackage[noadjust]{cite}
\usepackage{epigraph}
\usepackage[colorlinks=true,
            linkcolor=red,
            urlcolor=blue,
            citecolor=magenta]{hyperref}
\usepackage{color}
\usepackage{mathrsfs}
\allowdisplaybreaks
\theoremstyle{plain}
\newtheorem{thm}{Theorem}[section]
\newtheorem*{thm*}{Theorem}
\newtheorem{prop}{Proposition}[section]
\newtheorem*{prop*}{Proposition}
\newtheorem{cor}{Corollary}[section]
\newtheorem*{cor*}{Corollary}

\newtheorem*{lem*}{Lemma}
\theoremstyle{definition}
\newtheorem{defn}{Definition}[section]
\newtheorem*{defn*}{Definition}

\newtheorem*{exmp*}{Example}

\newtheorem*{exmps*}{Examples}

\newtheorem{rem}{Remark}[section]
\newtheorem*{rem*}{Remark}
\newtheorem{rems}{Remarks}[section]
\newtheorem*{rems*}{Remarks}

\newtheorem*{note*}{Note}
\newcommand{\N}{{\mathbb N}}
\newcommand{\Z}{{\mathbb Z}}
\newcommand{\R}{{\mathbb R}}
\newcommand{\C}{{\mathbb C}}

\DeclareMathOperator{\Rep}{Re\,}
\DeclareMathOperator{\Imp}{Im\,}
\DeclareMathOperator{\dist}{dist}
\DeclareMathOperator{\spa}{span}
\begin{document}
\title[On the Gevrey ultradifferentiability of weak solutions]
{On the Gevrey ultradifferentiability\\ 
of weak solutions\\
of an abstract evolution equation\\
with a scalar type spectral operator\\
on the open semi-axis}
\author[Marat V. Markin]{Marat V. Markin}
\address{
Department of Mathematics\newline
California State University, Fresno\newline
5245 N. Backer Avenue, M/S PB 108\newline
Fresno, CA 93740-8001, USA
}
\email{mmarkin@csufresno.edu}
\begin{abstract}
Given the abstract evolution equation
\[
y'(t)=Ay(t),\ t\ge 0,
\]
with \textit{scalar type spectral operator} $A$ in a complex Banach space, found are conditions \textit{necessary and sufficient} for all \textit{weak solutions} of the equation, which a priori need not be strongly differentiable, to be strongly Gevrey ultradifferentiable of order $\beta\ge 1$, in particular \textit{analytic} or \textit{entire}, on the open semi-axis $(0,\infty)$. Also, revealed is a certain interesting inherent smoothness improvement effect.
\end{abstract}
\maketitle
\section[Introduction]{Introduction}

We find conditions on a \textit{scalar type spectral operator} $A$ in a complex Banach space necessary and sufficient for all \textit{weak solutions} of the evolution equation
\begin{equation}\label{1}
y'(t)=Ay(t),\ t\ge 0,
\end{equation}
which a priori need not be strongly differentiable,
to be strongly Gevrey ultradifferentiable of order $\beta\ge 1$, in particular \textit{analytic}, on the \textit{open semi-axis} $(0,\infty)$ and reveal a certain interesting inherent smoothness improvement effect. 

The found results generalize the corresponding ones of
paper \cite{Markin2001(1)}, where similar consideration is given to equation \eqref{1} with a \textit{normal operator} $A$ in a complex Hilbert space, and the characterizations of the generation of Gevrey ultradifferentiable $C_0$-semigroups of Roumieu and Beurling types by scalar type spectral operators found in papers \cite{Markin2004(1),Markin2016} (see also \cite{Markin2008}). They also develop the discourse of papers \cite{Markin2011,Markin2018(4)}, in which the strong differentiability of the weak solutions of equation \eqref{1} on $[0,\infty)$ and $(0,\infty)$ and their strong Gevrey ultradifferentiability of order $\beta\ge 1$ on the \textit{closed semi-axis} $[0,\infty)$ are treated, respectively (cf. also \cite{Markin2018(3)}). 

\begin{defn}[Weak Solution]\label{ws}\ \\
Let $A$ be a densely defined closed linear operator in a Banach space $(X,\|\cdot\|)$. A strongly continuous vector function $y:[0,\infty)\rightarrow X$ is called a {\it weak solution} of equation \eqref{1} if, for any $g^* \in D(A^*)$,
\begin{equation*}
\dfrac{d}{dt}\langle y(t),g^*\rangle = \langle y(t),A^*g^* \rangle,\ t\ge 0,
\end{equation*}
where $D(\cdot)$ is the \textit{domain} of an operator, $A^*$ is the operator {\it adjoint} to $A$, and $\langle\cdot,\cdot\rangle$ is the {\it pairing} between
the space $X$ and its dual $X^*$ (cf. \cite{Ball}).
\end{defn}

\begin{rems}\label{remsws}\
\begin{itemize}
\item Due to the \textit{closedness} of $A$, the weak solution of \eqref{1} can be equivalently defined to be a strongly continuous vector function $y:[0,\infty)\mapsto X$ such that, for all $t\ge 0$,
\begin{equation*}
\int_0^ty(s)\,ds\in D(A)\ \text{and} \ y(t)=y(0)+A\int_0^ty(s)\,ds
\end{equation*}
and is also called a \textit{mild solution} (cf. {\cite[Ch. II, Definition 6.3]{Engel-Nagel}}, {\cite[Preliminaries]{Markin2018(2)}}).
\item Such a notion of \textit{weak solution}, which need not be differentiable in the strong sense, generalizes that of \textit{classical} one, strongly differentiable on $[0,\infty)$ and satisfying the equation in the traditional plug-in sense, the classical solutions being precisely the weak ones strongly differentiable on $[0,\infty)$.
\item When a closed densely defined linear operator $A$
in a complex Banach space $X$ generates a $C_0$-semigroup $\left\{T(t) \right\}_{t\ge 0}$ of  bounded linear operators (see, e.g., \cite{Hille-Phillips,Engel-Nagel}), i.e., the associated \textit{abstract Cauchy problem} (\textit{ACP})
\begin{equation}\label{ACP}
\begin{cases}
y'(t)=Ay(t),\ t\ge 0,\\
y(0)=f
\end{cases}
\end{equation}
is \textit{well-posed} (cf. {\cite[Ch. II, Definition 6.8]{Engel-Nagel}}), the weak solutions of equation \eqref{1} are the orbits
\begin{equation}\label{semigroup}
y(t)=T(t)f,\ t\ge 0,
\end{equation}
with $f\in X$ {\cite[Ch. II, Proposition 6.4]{Engel-Nagel}} (see also {\cite[Theorem]{Ball}}), whereas the classical ones are those with $f\in D(A)$
(see, e.g., {\cite[Ch. II, Proposition 6.3]{Engel-Nagel}}). 
\item In our consideration, the associated \textit{ACP} need not be \textit{well-posed}, i.e., the scalar type spectral operator $A$ need not generate a $C_0$-semigroup (cf. \cite{Markin2002(2)}).
\end{itemize} 
\end{rems} 

\section[Preliminaries]{Preliminaries}

Here, for the reader's convenience, we outline certain essential preliminaries.

\subsection{Scalar Type Spectral Operators}\ 

Henceforth, unless specified otherwise, $A$ is supposed to be a {\it scalar type spectral operator} in a complex Banach space $(X,\|\cdot\|)$ with strongly $\sigma$-additive \textit{spectral measure} (the \textit{resolution of the identity}) $E_A(\cdot)$ assigning to each Borel set $\delta$ of the complex plane $\C$ a projection operator $E_A(\delta)$ on $X$ and having the operator's \textit{spectrum} $\sigma(A)$ as its {\it support} \cite{Dunford1954,Survey58,Dun-SchIII}.

Observe that, on a complex finite-dimensional space, 
the scalar type spectral operators are all linear operators that furnish an \textit{eigenbasis} for the space (see, e.g., \cite{Survey58,Dun-SchIII}) and, in a complex Hilbert space, the scalar type spectral operators are precisely all those that are similar to the {\it normal} ones \cite{Wermer}.

Associated with a scalar type spectral operator in a complex Banach space is the {\it Borel operational calculus} analogous to that for a \textit{normal operator} in a complex Hilbert space \cite{Dun-SchII,Plesner}, which assigns to any Borel measurable function $F:\sigma(A)\to \C$ a scalar type spectral operator
\begin{equation*}
F(A):=\int\limits_{\sigma(A)} F(\lambda)\,dE_A(\lambda)
\end{equation*}
(see \cite{Survey58,Dun-SchIII}).

In particular,
\begin{equation*}
A^n=\int\limits_{\sigma(A)} \lambda^n\,dE_A(\lambda),\ n\in\Z_+,
\end{equation*}
($\Z_+:=\left\{0,1,2,\dots\right\}$ is the set of nonnegative integers, $A^0:=I$, $I$ is the \textit{identity operator} on $X$) and
\begin{equation}\label{exp}
e^{zA}:=\int\limits_{\sigma(A)} e^{z\lambda}\,dE_A(\lambda),\ z\in\C.
\end{equation}

The properties of the {\it spectral measure} and {\it operational calculus}, exhaustively delineated in \cite{Survey58,Dun-SchIII}, underlie the subsequent discourse. Here, we touch upon a few facts of particular importance.

Due to its {\it strong countable additivity}, the spectral measure $E_A(\cdot)$ is {\it bounded} \cite{Dun-SchI,Dun-SchIII}, i.e., there is such an $M\ge 1$ that, for any Borel set $\delta\subseteq \C$,
\begin{equation}\label{bounded}
\|E_A(\delta)\|\le M.
\end{equation}

Observe that the notation $\|\cdot\|$ is used here to designate the norm in the space $L(X)$ of all bounded linear operators on $X$. We adhere to this rather conventional economy of symbols in what follows also adopting the same notation for the norm in the dual space $X^*$. 

For any $f\in X$ and $g^*\in X^*$, the \textit{total variation measure} $v(f,g^*,\cdot)$ of the complex-valued Borel measure $\langle E_A(\cdot)f,g^* \rangle$ is a {\it finite} positive Borel measure with
\begin{equation}\label{tv}
v(f,g^*,\C)=v(f,g^*,\sigma(A))\le 4M\|f\|\|g^*\|
\end{equation}
(see, e.g., \cite{Markin2004(1),Markin2004(2)}).

Also (Ibid.), for a Borel measurable function $F:\C\to \C$, $f\in D(F(A))$, $g^*\in X^*$, and a Borel set $\delta\subseteq \C$,
\begin{equation}\label{cond(ii)}
\int\limits_\delta|F(\lambda)|\,dv(f,g^*,\lambda)
\le 4M\|E_A(\delta)F(A)f\|\|g^*\|.
\end{equation}
In particular, for $\delta=\sigma(A)$,
\begin{equation}\label{cond(i)}
\int\limits_{\sigma(A)}|F(\lambda)|\,d v(f,g^*,\lambda)\le 4M\|F(A)f\|\|g^*\|.
\end{equation}

Observe that the constant $M\ge 1$ in \eqref{tv}--\eqref{cond(i)} is from 
\eqref{bounded}.

Further, for a Borel measurable function $F:\C\to [0,\infty)$, a Borel set $\delta\subseteq \C$, a sequence $\left\{\Delta_n\right\}_{n=1}^\infty$ 
of pairwise disjoint Borel sets in $\C$, and 
$f\in X$, $g^*\in X^*$,
\begin{equation}\label{decompose}
\int\limits_{\delta}F(\lambda)\,dv(E_A(\cup_{n=1}^\infty \Delta_n)f,g^*,\lambda)
=\sum_{n=1}^\infty \int\limits_{\delta\cap\Delta_n}F(\lambda)\,dv(E_A(\Delta_n)f,g^*,\lambda).
\end{equation}

Indeed, since, for any Borel sets $\delta,\sigma\subseteq \C$,
\begin{equation*}
E_A(\delta)E_A(\sigma)=E_A(\delta\cap\sigma)
\end{equation*}
\cite{Survey58,Dun-SchIII}, 
for the total variation,
\begin{equation*}
v(E_A(\delta)f,g^*,\sigma)=v(f,g^*,\delta\cap\sigma).
\end{equation*}

Whence, due to the {\it nonnegativity} of $F(\cdot)$ (see, e.g., \cite{Halmos}),
\begin{multline*}
\int\limits_\delta F(\lambda)\,dv(E_A(\cup_{n=1}^\infty \Delta_n)f,g^*,\lambda)
=\int\limits_{\delta\cap\cup_{n=1}^\infty \Delta_n}F(\lambda)\,dv(f,g^*,\lambda)
\\
\ \
=\sum_{n=1}^\infty \int\limits_{\delta\cap\Delta_n}F(\lambda)\,dv(f,g^*,\lambda)
=\sum_{n=1}^\infty \int\limits_{\delta\cap\Delta_n}F(\lambda)\,dv(E_A(\Delta_n)f,g^*,\lambda).
\hfill
\end{multline*}

The following statement, allowing to characterize the domains of Borel measurable functions of a scalar type spectral operator in terms of positive Borel measures, is fundamental for our consideration.

\begin{prop}[{\cite[Proposition $3.1$]{Markin2002(1)}}]\label{prop}\ \\
Let $A$ be a scalar type spectral operator in a complex Banach space $(X,\|\cdot\|)$ with spectral measure $E_A(\cdot)$ and $F:\sigma(A)\to \C$ be a Borel measurable function. Then $f\in D(F(A))$ iff
\begin{enumerate}[label={(\roman*)}]
\item for each $g^*\in X^*$, 
$\displaystyle \int\limits_{\sigma(A)} |F(\lambda)|\,d v(f,g^*,\lambda)<\infty$ and
\item $\displaystyle \sup_{\{g^*\in X^*\,|\,\|g^*\|=1\}}
\int\limits_{\{\lambda\in\sigma(A)\,|\,|F(\lambda)|>n\}}
|F(\lambda)|\,dv(f,g^*,\lambda)\to 0,\ n\to\infty$,
\end{enumerate}
where $v(f,g^*,\cdot)$ is the total variation measure of $\langle E_A(\cdot)f,g^* \rangle$.
\end{prop} 

The succeeding key theorem provides a description of the weak solutions of equation \eqref{1} with a scalar type spectral operator $A$ in a complex Banach space.

\begin{thm}[{\cite[Theorem $4.2$]{Markin2002(1)}}]\label{GWS}\ \\
Let $A$ be a scalar type spectral operator in a complex Banach space $(X,\|\cdot\|)$. A vector function $y:[0,\infty) \to X$ is a weak solution 
of equation \eqref{1} iff there is an $\displaystyle f \in \bigcap_{t\ge 0}D(e^{tA})$ such that
\begin{equation}\label{expf}
y(t)=e^{tA}f,\ t\ge 0,
\end{equation}
the operator exponentials understood in the sense of the Borel operational calculus (see \eqref{exp}).
\end{thm}

\begin{rems}\
\begin{itemize}
\item Theorem \ref{GWS} generalizing {\cite[Theorem $3.1$]{Markin1999}}, its counterpart for a normal operator $A$ in a complex Hilbert space, in particular,
implies
\begin{itemize}
\item that the subspace $\bigcap_{t\ge 0}D(e^{tA})$ of all possible initial values of the weak solutions of equation \eqref{1} is the largest permissible for the exponential form given by \eqref{expf}, which highlights the naturalness of the notion of weak solution, and
\item that associated \textit{ACP} \eqref{ACP}, whenever solvable,  is solvable \textit{uniquely}.
\end{itemize} 
\item Observe that the initial-value subspace $\bigcap_{t\ge 0}D(e^{tA})$ of equation \eqref{1} is \textit{dense} in $X$ since it contains the subspace
\begin{equation*}
\bigcup_{\alpha>0}E_A(\Delta_\alpha)X,\ \text{where}\ \Delta_\alpha:=\left\{\lambda\in\C\,\middle|\,|\lambda|\le \alpha \right\},\ \alpha>0,
\end{equation*}
which is dense in $X$ and coincides with the class ${\mathscr E}^{\{0\}}(A)$ of \textit{entire} vectors of $A$ of \textit{exponential type} \cite{Markin2015,Radyno1983(1)}.
\item When a scalar type spectral operator $A$ in a complex Banach space generates a $C_0$-semigroup $\left\{T(t) \right\}_{t\ge 0}$, 
\[
T(t)=e^{tA}\ \text{and}\ D(e^{tA})=X,\ t\ge 0,
\]
\cite{Markin2002(2)}, and hence, Theorem \ref{GWS} is consistent with the well-known description of the weak solutions for this setup (see \eqref{semigroup}).
\end{itemize} 
\end{rems} 

Subsequently, the frequent terms {\it ``spectral measure"} and {\it ``operational calculus"} are abbreviated to {\it s.m.} and {\it o.c.}, respectively.

\subsection{Gevrey Classes of Functions}\label{GCF}\

\begin{defn}[Gevrey Classes of Functions]\ \\
Let $(X,\|\cdot\|)$ be a (real or complex) Banach space, $C^\infty(I,X)$ be the space of all $X$-valued functions strongly infinite differentiable on an interval $I\subseteq \R$, and $0\le \beta<\infty$.

The following subspaces of $C^\infty(I,X)$
\begin{align*}
{\mathscr E}^{\{\beta\}}(I,X):=\bigl\{g(\cdot)\in C^{\infty}(I, X) \bigm |&
\forall\, [a,b] \subseteq I\ \exists\, \alpha>0\ \exists\, c>0:
\\
&\max_{a \le t \le b}\|g^{(n)}(t)\| \le c\alpha^n [n!]^\beta,
\ n\in\Z_+\bigr\},\\
{\mathscr E}^{(\beta)}(I,X):= \bigl\{g(\cdot) \in C^{\infty}(I,X) \bigm |& 
\forall\, [a,b] \subseteq I\ \forall\, \alpha > 0 \ \exists\, c>0:
\\
&\max_{a \le t \le b}\|g^{(n)}(t)\| \le c\alpha^n [n!]^\beta,
\ n\in\Z_+\bigr\},
\end{align*}
are called the {\it $\beta$th-order Gevrey classes} of strongly ultradifferentiable vector functions on $I$ of {\it Roumieu} and {\it Beurling type}, respectively (see, e.g., \cite{Gevrey,Komatsu1,Komatsu2,Komatsu3}).
\end{defn}

\begin{rems}\
\begin{itemize}
\item In view of {\it Stirling's formula}, the 
sequence $\left\{ [n!]^\beta\right\}_{n=0}^\infty$ can be replaced with
$\left\{ n^{\beta n}\right\}_{n=0}^\infty$.
\item For $0\le\beta<\beta'<\infty$, the inclusions
\begin{equation*}
{\mathscr E}^{(\beta)}(I,X)\subseteq{\mathscr E}^{\{\beta\}}(I,X)
\subseteq {\mathscr E}^{(\beta')}(I,X)\subseteq
{\mathscr E}^{\{\beta'\}}(I,X)\subseteq C^{\infty}(I,X)
\end{equation*}
hold.
\item For $1<\beta<\infty$, the Gevrey classes
are \textit{non-quasianalytic} (see, e.g., \cite{Komatsu2}).
\item For $\beta=1$, ${\mathscr E}^{\{1\}}(I,X)$ 
is the class of all {\it analytic} on $I$, i.e., {\it analytically continuable} into complex neighborhoods of $I$, vector functions and ${\mathscr E}^{(1)}(I,X)$ is the class of all {\it entire}, i.e., allowing {\it entire} continuations, vector functions \cite{Mandel}.
\item For $0\le\beta<1$, the Gevrey class ${\mathscr E}^{\{\beta\}}(I,X)$ (${\mathscr E}^{(\beta)}(I,X)$) consists of all functions $g(\cdot)\in {\mathscr E}^{(1)}(I,X)$ such that, for some (any) $\gamma>0$, there is an $M>0$ for which
\begin{equation}\label{order}
\|g(z)\|\le Me^{\gamma|z|^{1/(1-\beta)}},\ z\in \C,
\end{equation}
\cite{Markin2001(2)} (see also \cite{Markin2019(1)}). In particular, for $\beta=0$, ${\mathscr E}^{\{0\}}(I,X)$ and ${\mathscr E}^{(0)}(I,X)$ are the classes of entire vector functions of \textit{exponential} and \textit{minimal exponential type}, respectively (see, e.g., \cite{Levin}).
\end{itemize} 
\end{rems} 

\subsection{Gevrey Classes of Vectors}\

One can consider the Gevrey classes in a more general sense. 

\begin{defn}[Gevrey Classes of Vectors]\ \\
Let $(A,D(A))$ be a densely defined closed linear operator in a (real or complex) Banach space $(X,\|\cdot\|)$, $0\le \beta<\infty$, and
\begin{equation*}
C^{\infty}(A):=\bigcap_{n=0}^{\infty}D(A^n)
\end{equation*}
be the subspace of \textit{infinite differentiable vectors} of $A$.

The following subspaces of $C^{\infty}(A)$
\begin{align*}
{\mathscr E}^{\{\beta\}}(A)&:=\left\{x\in C^{\infty}(A)\, \middle |\, 
\exists\, \alpha>0\ \exists\, c>0:
\|A^nx\| \le c\alpha^n [n!]^\beta,\ n\in\Z_+ \right\},\\
{\mathscr E}^{(\beta)}(A)&:=\left\{x \in C^{\infty}(A)\, \middle|\,\forall\, \alpha > 0 \ \exists\, c>0:
\|A^nx\| \le c\alpha^n [n!]^\beta,\ n\in\Z_+ \right\}
\end{align*}
are called the \textit{$\beta$th-order Gevrey classes} of ultradifferentiable vectors of $A$ of \textit{Roumieu} and \textit{Beurling type}, respectively (see, e.g., \cite{GorV83,Gor-Knyaz,book}).
\end{defn}

\begin{samepage}
\begin{rems}\
\begin{itemize}
\item In view of {\it Stirling's formula}, the 
sequence $\left\{ [n!]^\beta\right\}_{n=0}^\infty$ can be replaced with
$\left\{ n^{\beta n}\right\}_{n=0}^\infty$.
\item For $0\le\beta<\beta'<\infty$, the inclusions
\begin{equation*}
{\mathscr E}^{(\beta)}(A)\subseteq{\mathscr E}^{\{\beta\}}(A)
\subseteq {\mathscr E}^{(\beta')}(A)\subseteq
{\mathscr E}^{\{\beta'\}}(A)\subseteq C^{\infty}(A)
\end{equation*}
hold.
\item In particular, ${\mathscr E}^{\{1\}}(A)$ and ${\mathscr E}^{(1)}(A)$ are the classes of {\it analytic} and {\it entire} vectors of $A$, respectively \cite{Goodman,Nelson} and ${\mathscr E}^{\{0\}}(A)$ and ${\mathscr E}^{(0)}(A)$ are the classes of \textit{entire} vectors of $A$ of \textit{exponential} and \textit{minimal exponential type}, respectively (see, e.g., \cite{Radyno1983(1),Gor-Knyaz}).
\item In view of the \textit{closedness} of $A$, it is easily seen that the class ${\mathscr E}^{(1)}(A)$ forms the subspace of the initial values $f\in X$ generating the (classical) solutions of \eqref{1}, which are entire vector functions represented by the power series
\begin{equation*}
\sum_{n=0}^\infty \dfrac{t^n}{n!}A^nf,\ t\ge 0,
\end{equation*}
the classes ${\mathscr E}^{\{\beta\}}(A)$ and ${\mathscr E}^{(\beta)}(A)$ with $0\le\beta<1$ being the subspaces of such initial values for which the solutions satisfy growth estimate \eqref{order} with some (any) $\gamma>0$ and some $M>0$, respectively (cf. \cite{Levin}).
\end{itemize} 
\end{rems} 
\end{samepage} 

As is shown in \cite{GorV83} (see also \cite{Gor-Knyaz,book}), if $0<\beta<\infty$, for a {\it normal operator} $A$ in a complex Hilbert space,
\begin{equation}\label{GC}
{\mathscr E}^{\{\beta\}}(A)=\bigcup_{t>0} D(e^{t|A|^{1/\beta}})\ \text{and}\ 
{\mathscr E}^{(\beta)}(A)=\bigcap_{t>0} D(e^{t|A|^{1/\beta}}),
\end{equation}
the operator exponentials $e^{t|A|^{1/\beta}}$, $t>0$, understood in the sense of the Borel operational calculus (see, e.g., \cite{Dun-SchII,Plesner}).

In \cite{Markin2004(2),Markin2015}, descriptions \eqref{GC} are extended  to \textit{scalar type spectral operators} in a complex Banach space, in which form they are basic for our discourse. In \cite{Markin2015}, similar nature descriptions of the classes ${\mathscr E}^{\{0\}}(A)$ and ${\mathscr E}^{(0)}(A)$ ($\beta=0$), known for a normal operator $A$ in a complex Hilbert space (see, e.g., \cite{Gor-Knyaz}), are also generalized to scalar type spectral operators in a complex Banach space. In particular {\cite[Theorem $5.1$]{Markin2015}},
\[
{\mathscr E}^{\{0\}}(A)=\bigcup_{\alpha>0}E_A(\Delta_\alpha)X,
\] 
where
\begin{equation*}
\Delta_\alpha:=\left\{\lambda\in\C\,\middle|\,|\lambda|\le \alpha \right\},\ \alpha>0.
\end{equation*}

We also need the following characterization of a particular weak solution's of equation \eqref{1} with a scalar type spectral operator $A$ in a complex Banach space being strongly Gevrey ultradifferentiable on a subinterval $I$ of $[0,\infty)$.

\begin{prop}[{\cite[Proposition $3.2$]{Markin2018(4)}}]\label{particular}\ \\
Let $A$ be a scalar type spectral operator in a complex Banach space $(X,\|\cdot\|)$, $0\le \beta<\infty$, and $I$ be a subinterval of $[0,\infty)$. Then the restriction of
a weak solution $y(\cdot)$ of equation \eqref{1} to $I$ belongs to the Gevrey class ${\mathscr E}^{\{\beta\}}(I,X)$
\textup{(${\mathscr E}^{(\beta)}(I,X)$)} iff, for each $t\in I$, 
\begin{equation*}
y(t) \in {\mathscr E}^{\{\beta\}}(A)
\ \textup{(${\mathscr E}^{(\beta)}(A)$, respectively)},
\end{equation*}
in which case, for every $n\in\N$,
\begin{equation*}
y^{(n)}(t)=A^ny(t),\ t \in I.
\end{equation*}
\end{prop}

\section{Gevrey Ultradifferentiability of Weak Solutions}

The case of the strong Gevrey ultradifferentiability of the weak solutions of equation \eqref{1} with a scalar type spectral operator in a complex Banach space on the \textit{open semi-axis} $(0,\infty)$, similarly to the analogous setup with a normal operator $A$ in a complex Hilbert space \cite{Markin2001(1)}, significantly differs from its counterpart over the \textit{closed semi-axis} $[0,\infty)$ studied in \cite{Markin2018(4)}. 

First, let us consider the Roumieu type strong Gevrey ultradifferentiability of order $\beta\ge 1$.

\begin{thm}\label{open1}
Let $A$ be a scalar type spectral operator in a complex Banach space $(X,\|\cdot\|)$ with spectral measure $E_A(\cdot)$ and $ 1\le \beta<\infty$.
Every weak solution of equation \eqref{1} belongs to the $\beta$th-order Roumieu type Gevrey class 
${\mathscr E}^{\{\beta \}}\left((0,\infty),X\right)$ iff there exist
$b_+>0$ and $ b_->0$ such that the set $\sigma(A)\setminus {\mathscr P}^\beta_{b_-,b_+}$,
where
\begin{equation*}
{\mathscr P}^\beta_{b_-,b_+}:=\left\{\lambda \in \C\, \middle|\,
\Rep\lambda \le -b_-|\Imp\lambda|^{1/\beta} 
\ \text{or}\ 
\Rep\lambda \ge b_+|\Imp\lambda|^{1/\beta} \right\},
\end{equation*}
is bounded 
(see Fig. \ref{fig:graph4}).

\begin{figure}[h]
\centering
\includegraphics[height=2in]{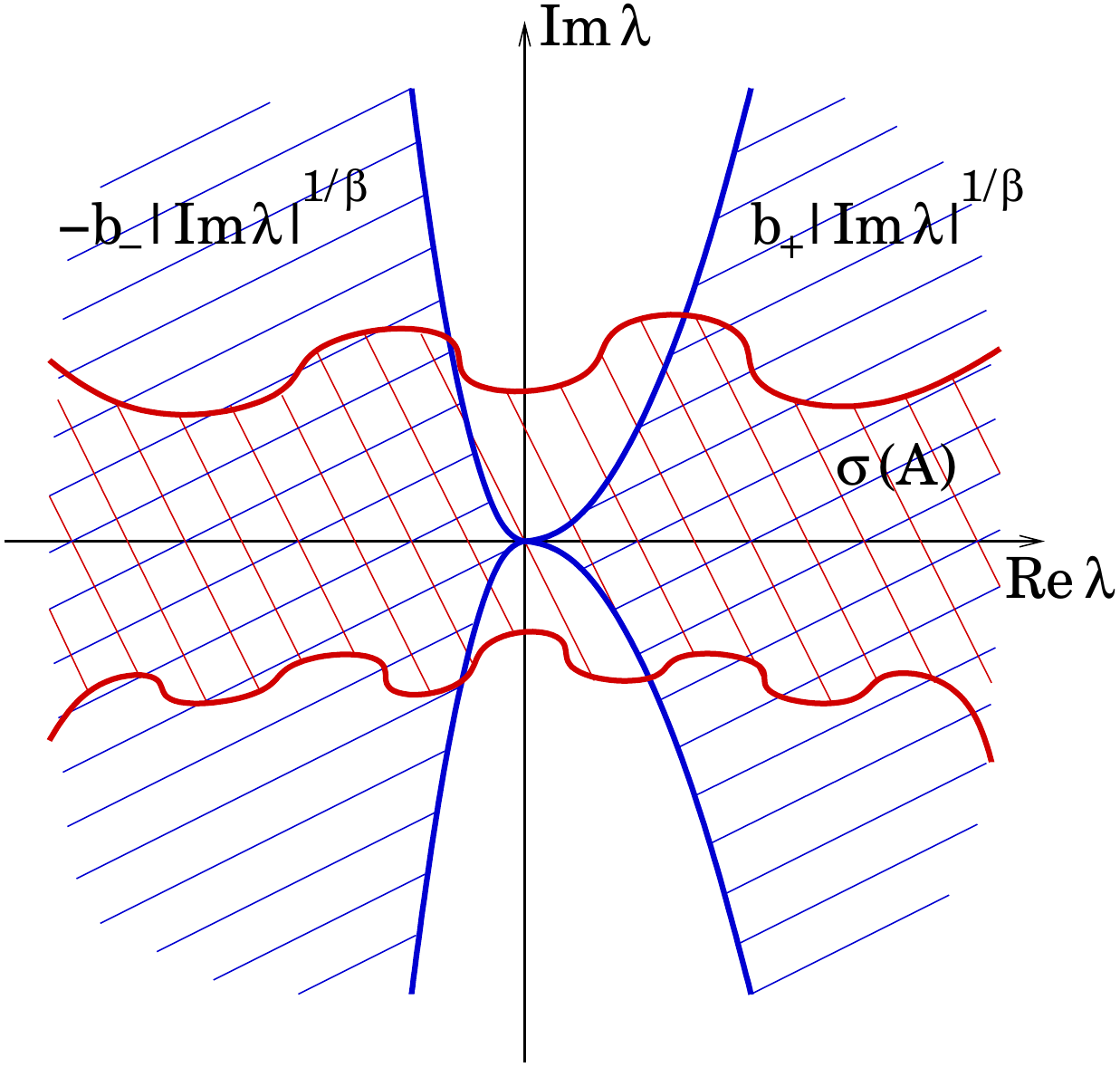}
\caption[]{}
\label{fig:graph4}
\end{figure}
\end{thm}

\begin{proof}\

\textit{``If"} Part. Suppose that there exist $b_+>0$ and $ b_->0$ such that the set $\sigma(A)\setminus {\mathscr P}^\beta_{b_-,b_+}$
is \textit{bounded} and let $y(\cdot)$ be an arbitrary weak solution of equation \eqref{1}. 

By Theorem \ref{GWS}, 
\begin{equation*}
y(t)=e^{tA}f,\ t\ge 0,\ \text{with some}\
f \in \bigcap_{t\ge 0}D(e^{tA}).
\end{equation*}

Our purpose is to show that $y(\cdot)\in {\mathscr E}^{\{\beta \}}\left((0,\infty),X\right)$, which, by Proposition \ref{particular} and \eqref{GC}, is attained by showing that, for each $t>0$,
\[
y(t)\in {\mathscr E}^{\{\beta \}}\left(A\right)
=\bigcup_{s>0} D(e^{s|A|^{1/\beta}}).
\]

Let us proceed by proving that, for each $t>0$, there exists an $s>0$ such that
\[
y(t)\in D(e^{s|A|^{1/\beta}})
\] 
via Proposition \ref{prop}.

For an arbitrary $t>0$, let us set 
\begin{equation}\label{s}
s:=t(1+b_-^{-\beta})^{-1/\beta}>0.
\end{equation}

Then, for any $g^*\in X^*$,
\begin{multline}\label{first}
\int\limits_{\sigma(A)}e^{s|\lambda|^{1/\beta}}e^{t\Rep\lambda}\,dv(f,g^*,\lambda)
=\int\limits_{\sigma(A)\setminus{\mathscr P}_{b_-,b_+}^\beta}e^{s|\lambda|^{1/\beta}}e^{t\Rep\lambda}\,dv(f,g^*,\lambda)
\\
\shoveleft{
+\int\limits_{\left\{\lambda\in \sigma(A)\cap{\mathscr P}_{b_-,b_+}^\beta\,\middle|\,-1<\Rep\lambda<1 \right\}}e^{s|\lambda|^{1/\beta}}e^{t\Rep\lambda}\,dv(f,g^*,\lambda)
}\\
\shoveleft{
+\int\limits_{\left\{\lambda\in \sigma(A)\cap{\mathscr P}_{b_-,b_+}^\beta\,\middle|\,\Rep\lambda\ge 1 \right\}}e^{s|\lambda|^{1/\beta}}e^{t\Rep\lambda}\,dv(f,g^*,\lambda)
}\\
\hspace{1.2cm}
+\int\limits_{\left\{\lambda\in \sigma(A)\cap{\mathscr P}_{b_-,b_+}^\beta\,\middle|\,\Rep\lambda\le -1 \right\}}e^{s|\lambda|^{1/\beta}}e^{t\Rep\lambda}\,dv(f,g^*,\lambda)<\infty.
\hfill
\end{multline}

Indeed, 
\[
\int\limits_{\sigma(A)\setminus{\mathscr P}_{b_-,b_+}^\beta}e^{s|\lambda|^{1/\beta}}e^{t\Rep\lambda}\,dv(f,g^*,\lambda)<\infty
\]
and
\[
\int\limits_{\left\{\lambda\in \sigma(A)\cap{\mathscr P}_{b_-,b_+}^\beta\,\middle|\,-1<\Rep\lambda<1 \right\}}e^{s|\lambda|^{1/\beta}}e^{t\Rep\lambda}\,dv(f,g^*,\lambda)<\infty
\]
due to the boundedness of the sets
\[
\sigma(A)\setminus{\mathscr P}_{b_-,b_+}^\beta\ \text{and}\
\left\{\lambda\in \sigma(A)\cap{\mathscr P}_{b_-,b_+}^\beta\;\middle|\;-1<\Rep\lambda<1 \right\},
\]
the continuity of the integrated function 
on $\C$, and the finiteness of the measure $v(f,g^*,\cdot)$.

Further, for an arbitrary $t>0$, $s>0$ chosen as in \eqref{s}, and any $g^*\in X^*$,
\begin{multline}\label{interm}
\int\limits_{\left\{\lambda\in \sigma(A)\cap{\mathscr P}_{b_-,b_+}^\beta\,\middle|\,\Rep\lambda\ge 1 \right\}}e^{s|\lambda|^{1/\beta}}e^{t\Rep\lambda}\,dv(f,g^*,\lambda)
\\
\shoveleft{
\le\int\limits_{\left\{\lambda\in \sigma(A)\cap{\mathscr P}_{b_-,b_+}^\beta\,\middle|\,\Rep\lambda\ge 1 \right\}}e^{s\left[|\Rep\lambda|+|\Imp\lambda|\right]^{1/\beta}}e^{t\Rep\lambda}\,dv(f,g^*,\lambda)
}\\
\hfill
\text{since, for $\lambda\in\sigma(A)\cap{\mathscr P}_{b_-,b_+}^\beta$ with $\Rep\lambda\ge 1$, $b_+^{-\beta}\Rep\lambda^\beta\ge |\Imp\lambda|$;}
\\
\shoveleft{
\le 
\int\limits_{\left\{\lambda\in \sigma(A)\cap{\mathscr P}_{b_-,b_+}^\beta\,\middle|\,\Rep\lambda\ge 1 \right\}}e^{s\left[\Rep\lambda+b_+^{-\beta}\Rep\lambda^\beta\right]^{1/\beta}}e^{t\Rep\lambda}\,dv(f,g^*,\lambda)
}\\
\hfill
\text{since, in view of $\Rep\lambda\ge 1$ and $\beta\ge 1$, $\Rep\lambda^\beta\ge\Rep\lambda$;}
\\
\shoveleft{
\le 
\int\limits_{\left\{\lambda\in \sigma(A)\cap{\mathscr P}_{b_-,b_+}^\beta\,\middle|\,\Rep\lambda\ge 1 \right\}}e^{s\left(1+b_+^{-\beta}\right)^{1/\beta}\Rep\lambda}e^{t\Rep\lambda}\,dv(f,g^*,\lambda)
}\\
\shoveleft{
= \int\limits_{\left\{\lambda\in \sigma(A)\cap{\mathscr P}_{b_-,b_+}^\beta\,\middle|\,\Rep\lambda\ge 1 \right\}}e^{\left[s\left(1+b_+^{-\beta}\right)^{1/\beta}+t\right]\Rep\lambda}\,dv(f,g^*,\lambda)
}\\
\hfill
\text{since $f\in \bigcap\limits_{t\ge 0}D(e^{tA})$, by Proposition \ref{prop};}
\\
\hspace{1.2cm}
<\infty. 
\hfill
\end{multline}

Observe that, for the finiteness of the three preceding integrals, the special choice of $s>0$ is superfluous.

Finally, for an arbitrary $t>0$, $s>0$ chosen as in \eqref{s}, and any $g^*\in X^*$,
\begin{multline}\label{interm2}
\int\limits_{\left\{\lambda\in \sigma(A)\cap{\mathscr P}_{b_-,b_+}^\beta\,\middle|\,\Rep\lambda\le -1 \right\}}e^{s|\lambda|^{1/\beta}}e^{t\Rep\lambda}\,dv(f,g^*,\lambda)
\\
\shoveleft{
\le\int\limits_{\left\{\lambda\in \sigma(A)\cap{\mathscr P}_{b_-,b_+}^\beta\,\middle|\,\Rep\lambda\le -1 \right\}}e^{s\left[|\Rep\lambda|+|\Imp\lambda|\right]^{1/\beta}}e^{t\Rep\lambda}\,dv(f,g^*,\lambda)
}\\
\hfill
\text{since, for $\lambda\in\sigma(A)\cap{\mathscr P}_{b_-,b_+}^\beta$ with $\Rep\lambda\le -1$, $b_-^{-\beta}(-\Rep\lambda)^\beta\ge |\Imp\lambda|$;}
\\
\shoveleft{
\le 
\int\limits_{\left\{\lambda\in \sigma(A)\cap{\mathscr P}_{b_-,b_+}^\beta\,\middle|\,\Rep\lambda\le -1 \right\}}e^{s\left[-\Rep\lambda+b_-^{-\beta}(-\Rep\lambda)^\beta\right]^{1/\beta}}e^{t\Rep\lambda}\,dv(f,g^*,\lambda)
}\\
\hfill
\text{since, in view of $-\Rep\lambda\ge 1$ and $\beta\ge 1$, $(-\Rep\lambda)^\beta\ge-\Rep\lambda$;}
\\
\shoveleft{
\le 
\int\limits_{\left\{\lambda\in \sigma(A)\cap{\mathscr P}_{b_-,b_+}^\beta\,\middle|\,\Rep\lambda\le -1 \right\}}e^{s\left(1+b_-^{-\beta}\right)^{1/\beta}(-\Rep\lambda)}e^{t\Rep\lambda}\,dv(f,g^*,\lambda)
}\\
\shoveleft{
= \int\limits_{\left\{\lambda\in \sigma(A)\cap{\mathscr P}_{b_-,b_+}^\beta\,\middle|\,\Rep\lambda\le -1 \right\}}e^{\left[t-s\left(1+b_-^{-\beta}\right)^{1/\beta}\right]\Rep\lambda}\,dv(f,g^*,\lambda)
}\\
\hfill
\text{since $s:=t(1+b_-^{-\beta})^{-1/\beta}$ (see \eqref{s});}
\\
\shoveleft{
= \int\limits_{\left\{\lambda\in \sigma(A)\cap{\mathscr P}_{b_-,b_+}^\beta\,\middle|\,\Rep\lambda\le -1 \right\}}1\,dv(f,g^*,\lambda)
\le \int\limits_{\sigma(A)}1\,dv(f,g^*,\lambda)
=v(f,g^*,\sigma(A))
}\\
\hfill
\text{by \eqref{tv};}
\\
\hspace{1.2cm}
\le 4M\|f\|\|g^*\|<\infty. 
\hfill
\end{multline}

Also, for an arbitrary $t>0$, $s>0$ chosen as in \eqref{s}, and any $n\in\N$,
\begin{multline}\label{second}
\sup_{\{g^*\in X^*\,|\,\|g^*\|=1\}}
\int\limits_{\left\{\lambda\in\sigma(A)\,\middle|\,e^{s|\lambda|^{1/\beta}}e^{t\Rep\lambda}>n\right\}}
e^{s|\lambda|^{1/\beta}}e^{t\Rep\lambda}\,dv(f,g^*,\lambda)
\\
\shoveleft{
\le \sup_{\{g^*\in X^*\,|\,\|g^*\|=1\}}
\int\limits_{\left\{\lambda\in\sigma(A)\setminus{\mathscr P}_{b_-,b_+}^\beta\,\middle|\,e^{s|\lambda|^{1/\beta}}e^{t\Rep\lambda}>n\right\}}e^{s|\lambda|^{1/\beta}}e^{t\Rep\lambda}\,dv(f,g^*,\lambda)
}\\
\shoveleft{
+ \sup_{\{g^*\in X^*\,|\,\|g^*\|=1\}}
\int\limits_{\left\{\lambda\in\sigma(A)\cap{\mathscr P}_{b_-,b_+}^\beta\,\middle|\,-1<\Rep\lambda<1,\, e^{s|\lambda|^{1/\beta}}e^{t\Rep\lambda}>n\right\}}e^{s|\lambda|^{1/\beta}}e^{t\Rep\lambda}\,dv(f,g^*,\lambda)
}\\
\shoveleft{
+ \sup_{\{g^*\in X^*\,|\,\|g^*\|=1\}}
\int\limits_{\left\{\lambda\in\sigma(A)\cap{\mathscr P}_{b_-,b_+}^\beta\,\middle|\,\Rep\lambda\ge 1,\, e^{s|\lambda|^{1/\beta}}e^{t\Rep\lambda}>n\right\}}e^{s|\lambda|^{1/\beta}}e^{t\Rep\lambda}\,dv(f,g^*,\lambda)
}\\
\shoveleft{
+ \sup_{\{g^*\in X^*\,|\,\|g^*\|=1\}}
\int\limits_{\left\{\lambda\in\sigma(A)\cap{\mathscr P}_{b_-,b_+}^\beta\,\middle|\,\Rep\lambda\le -1,\, e^{s|\lambda|^{1/\beta}}e^{t\Rep\lambda}>n\right\}}e^{s|\lambda|^{1/\beta}}e^{t\Rep\lambda}\,dv(f,g^*,\lambda)
}\\
\hspace{1.2cm}
\to 0,\ n\to\infty.
\hfill
\end{multline}

Indeed, since, due to the boundedness of the sets
\[
\sigma(A)\setminus{\mathscr P}_{b_-,b_+}^\beta\ \text{and}\
\left\{\lambda\in\sigma(A)\cap{\mathscr P}_{b_-,b_+}^\beta\,\middle|\,-1<\Rep\lambda<1\right\}
\]
and the continuity of the integrated function on $\C$,
the sets
\[
\left\{\lambda\in\sigma(A)\setminus{\mathscr P}_{b_-,b_+}^\beta\,\middle|\,e^{s|\lambda|^{1/\beta}}e^{t\Rep\lambda}>n\right\}
\]
and 
\[
\left\{\lambda\in\sigma(A)\cap{\mathscr P}_{b_-,b_+}^\beta\,\middle|\,-1<\Rep\lambda<1,\, e^{s|\lambda|^{1/\beta}}e^{t\Rep\lambda}>n\right\}
\]
are \textit{empty} for all sufficiently large $n\in \N$,
we immediately infer that, for any $t>0$ and $s>0$ chosen as in \eqref{s},
\[
\lim_{n\to\infty}\sup_{\{g^*\in X^*\,|\,\|g^*\|=1\}}
\int\limits_{\left\{\lambda\in\sigma(A)\setminus{\mathscr P}_{b_-,b_+}^\beta\,\middle|\,e^{s|\lambda|^{1/\beta}}e^{t\Rep\lambda}>n\right\}}e^{s|\lambda|^{1/\beta}}e^{t\Rep\lambda}\,dv(f,g^*,\lambda)=0
\]
and
\[
\lim_{n\to\infty}\sup_{\{g^*\in X^*\,|\,\|g^*\|=1\}}
\int\limits_{\left\{\lambda\in\sigma(A)\cap{\mathscr P}_{b_-,b_+}^\beta\,\middle|\,-1<\Rep\lambda<1,\, e^{s|\lambda|^{1/\beta}}e^{t\Rep\lambda}>n\right\}}e^{s|\lambda|^{1/\beta}}e^{t\Rep\lambda}\,dv(f,g^*,\lambda)
=0.
\]

Further, for an arbitrary $t>0$, $s>0$ chosen as in \eqref{s}, and any $n\in\N$,
\begin{multline*}
\sup_{\{g^*\in X^*\,|\,\|g^*\|=1\}}
\int\limits_{\left\{\lambda\in\sigma(A)\cap{\mathscr P}_{b_-,b_+}^\beta\,\middle|\,\Rep\lambda\ge 1,\, e^{s|\lambda|^{1/\beta}}e^{t\Rep\lambda}>n\right\}}e^{s|\lambda|^{1/\beta}}e^{t\Rep\lambda}\,dv(f,g^*,\lambda)
\\
\hfill
\text{as in \eqref{interm};}
\\
\shoveleft{
\le \sup_{\{g^*\in X^*\,|\,\|g^*\|=1\}}
\int\limits_{\left\{\lambda\in\sigma(A)\cap{\mathscr P}_{b_-,b_+}^\beta\,\middle|\,\Rep\lambda\ge 1,\, e^{s|\lambda|^{1/\beta}}e^{t\Rep\lambda}>n\right\}}e^{\left[s\left(1+b_+^{-\beta}\right)^{1/\beta}+t\right]\Rep\lambda}\,dv(f,g^*,\lambda)
}\\
\hfill
\text{since $f\in \bigcap\limits_{t\ge 0}D(e^{tA})$, by \eqref{cond(ii)};}
\\
\shoveleft{
\le \sup_{\{g^*\in X^*\,|\,\|g^*\|=1\}}
}\\
\shoveleft{
4M\left\|E_A\left(\left\{\lambda\in\sigma(A)\cap{\mathscr P}_{b_-,b_+}^\beta\,\middle|\,\Rep\lambda\ge 1,\, e^{s|\lambda|^{1/\beta}}e^{t\Rep\lambda}>n\right\}\right)
e^{\left[s\left(1+b_+^{-\beta}\right)^{1/\beta}+t\right]A}f\right\|\|g^*\|
}\\
\shoveleft{
\le 4M\left\|E_A\left(\left\{\lambda\in\sigma(A)\cap{\mathscr P}_{b_-,b_+}^\beta\,\middle|\,\Rep\lambda\ge 1,\, e^{s|\lambda|^{1/\beta}}e^{t\Rep\lambda}>n\right\}\right)
e^{\left[s\left(1+b_+^{-\beta}\right)^{1/\beta}+t\right]A}f\right\|
}\\
\hfill
\text{by the strong continuity of the {\it s.m.};}
\\
\ \
\to 4M\left\|E_A\left(\emptyset\right)e^{\left[s\left(1+b_+^{-\beta}\right)^{1/\beta}+t\right]A}f\right\|=0,\ n\to\infty.
\hfill
\end{multline*}

Finally, for an arbitrary $t>0$, $s>0$ chosen as in \eqref{s}, and any $n\in\N$,
\begin{multline*}
\sup_{\{g^*\in X^*\,|\,\|g^*\|=1\}}
\int\limits_{\left\{\lambda\in\sigma(A)\cap{\mathscr P}_{b_-,b_+}^\beta\,\middle|\,\Rep\lambda\le -1,\, e^{s|\lambda|^{1/\beta}}e^{t\Rep\lambda}>n\right\}}e^{s|\lambda|^{1/\beta}}e^{t\Rep\lambda}\,dv(f,g^*,\lambda)
\\
\hfill
\text{as in \eqref{interm2};}
\\
\shoveleft{
\le \sup_{\{g^*\in X^*\,|\,\|g^*\|=1\}}
\int\limits_{\left\{\lambda\in\sigma(A)\cap{\mathscr P}_{b_-,b_+}^\beta\,\middle|\,\Rep\lambda\le -1,\, e^{s|\lambda|^{1/\beta}}e^{t\Rep\lambda}>n\right\}}e^{\left[t-s\left(1+b_-^{-\beta}\right)^{1/\beta}\right]\Rep\lambda}\,dv(f,g^*,\lambda)
}\\
\hfill
\text{since $s:=t(1+b_-^{-\beta})^{-1/\beta}$ (see \eqref{s});}
\\
\shoveleft{
=\sup_{\{g^*\in X^*\,|\,\|g^*\|=1\}}
\int\limits_{\left\{\lambda\in\sigma(A)\cap{\mathscr P}_{b_-,b_+}^\beta\,\middle|\,\Rep\lambda\le -1,\, e^{s|\lambda|^{1/\beta}}e^{t\Rep\lambda}>n\right\}}1\,dv(f,g^*,\lambda)
}\\
\hfill
\text{by \eqref{cond(ii)};}
\\
\shoveleft{
\le \sup_{\{g^*\in X^*\,|\,\|g^*\|=1\}}
4M\left\|E_A\left(\left\{\lambda\in\sigma(A)\cap{\mathscr P}_{b_-,b_+}^\beta\,\middle|\,\Rep\lambda\le -1,\, e^{s|\lambda|^{1/\beta}}e^{t\Rep\lambda}>n\right\}\right)
f\right\|\|g^*\|
}\\
\shoveleft{
\le 4M\left\|E_A\left(\left\{\lambda\in\sigma(A)\cap{\mathscr P}_{b_-,b_+}^\beta\,\middle|\,\Rep\lambda\le -1,\, e^{s|\lambda|^{1/\beta}}e^{t\Rep\lambda}>n\right\}\right)
f\right\|
}\\
\hfill
\text{by the strong continuity of the {\it s.m.};}
\\
\ \
\to 4M\left\|E_A\left(\emptyset\right)f\right\|=0,\ n\to\infty.
\hfill
\end{multline*}

By Proposition \ref{prop} and the properties of the \textit{o.c.} (see {\cite[Theorem XVIII.$2.11$ (f)]{Dun-SchIII}}), \eqref{first} and \eqref{second} jointly imply that, for any $t>0$,
\[
f\in D(e^{s|A|^{1/\beta}}e^{tA})
\]
with $s:=t(1+b_-^{-\beta})^{-1/\beta}>0$, and hence, in view of \eqref{GC}, for each $t>0$, 
\begin{equation*}
y(t)=e^{tA}f\in \bigcup_{s>0} D(e^{s|A|^{1/\beta}})
={\mathscr E}^{\{\beta\}}(A).
\end{equation*}

By Proposition \ref{particular}, we infer that
\begin{equation*}
y(\cdot) \in {\mathscr E}^{\{\beta\}}((0,\infty),X),
\end{equation*}
which completes the proof of the \textit{``if"} part.

\medskip
\textit{``Only if"} part. Let us prove this part {\it by contrapositive} assuming that, for any $b_+>0$ and $b_->0$, the set 
$\sigma(A)\setminus {\mathscr P}_{b_-,b_+}^\beta$ is \textit{unbounded}. In particular, this means that, for any $n\in \N$, unbounded is the set
\begin{equation*}
\sigma(A)\setminus {\mathscr P}^\beta_{n^{-1},n^{-2}}=
\left\{\lambda \in \sigma(A)\,\middle| 
-n^{-1}|\Imp\lambda|^{1/\beta}<\Rep\lambda < n^{-2}|\Imp\lambda|^{1/\beta}\right\}.
\end{equation*} 

Hence, we can choose a sequence $\left\{\lambda_n\right\}_{n=1}^\infty$  
of points in the complex plane as follows:
\begin{equation*}
\begin{split}
&\lambda_n \in \sigma(A),\ n\in \N,\\
&-n^{-1}|\Imp\lambda_n|^{1/\beta}<\Rep\lambda_n <n^{-2}|\Imp\lambda_n|^{1/\beta},\ n\in \N,\\
&\lambda_0:=0,\ |\lambda_n|>\max\left[n,|\lambda_{n-1}|\right],\ n\in \N.\\
\end{split}
\end{equation*}

The latter implies, in particular, that the points $\lambda_n$, $n\in\N$, are \textit{distinct} ($\lambda_i \neq \lambda_j$, $i\neq j$).

Since, for each $n\in \N$, the set
\begin{equation*}
\left\{ \lambda \in {\mathbb C}\,\middle|\, 
-n^{-1}|\Imp\lambda|^{1/\beta}<\Rep\lambda <n^{-2}|\Imp\lambda|^{1/\beta},\
|\lambda|>\max\bigl[n,|\lambda_{n-1}|\bigr]\right\}
\end{equation*}
is {\it open} in $\C$, along with the point $\lambda_n$, it contains an {\it open disk}
\begin{equation*}
\Delta_n:=\left\{\lambda \in \C\, \middle|\,|\lambda-\lambda_n|<\varepsilon_n \right\}
\end{equation*}
centered at $\lambda_n$ of some radius $\varepsilon_n>0$, i.e., for each $\lambda \in \Delta_n$,
\begin{equation}\label{disks1}
-n^{-1}|\Imp\lambda|^{1/\beta}<\Rep\lambda < n^{-2}|\Imp\lambda|^{1/\beta}\ \text{and}\ |\lambda|>\max\bigl[n,|\lambda_{n-1}|\bigr].
\end{equation}

Furthermore, we can regard the radii of the disks to be small enough so that
\begin{equation}\label{radii1}
\begin{split}
&0<\varepsilon_n<\dfrac{1}{n},\ n\in\N,\ \text{and}\\
&\Delta_i \cap \Delta_j=\emptyset,\ i\neq j
\quad \text{(i.e., the disks are {\it pairwise disjoint})}.
\end{split}
\end{equation}

Whence, by the properties of the {\it s.m.}, 
\begin{equation*}
E_A(\Delta_i)E_A(\Delta_j)=0,\ i\neq j,
\end{equation*}
where $0$ stands for the \textit{zero operator} on $X$.

Observe also, that the subspaces $E_A(\Delta_n)X$, $n\in \N$, are \textit{nontrivial} since
\[
\Delta_n \cap \sigma(A)\neq \emptyset,\ n\in\N,
\]
with $\Delta_n$ being an \textit{open set} in $\C$. 

By choosing a unit vector $e_n\in E_A(\Delta_n)X$ for each $n\in\N$, we obtain a sequence 
$\left\{e_n\right\}_{n=1}^\infty$ in $X$ such that
\begin{equation}\label{ortho1}
\|e_n\|=1,\ n\in\N,\ \text{and}\ E_A(\Delta_i)e_j=\delta_{ij}e_j,\ i,j\in\N,
\end{equation}
where $\delta_{ij}$ is the \textit{Kronecker delta}.

As is easily seen, \eqref{ortho1} implies that the vectors $e_n$, $n\in \N$, are \textit{linearly independent}.

Furthermore, there exists an $\varepsilon>0$ such that
\begin{equation}\label{dist1}
d_n:=\dist\left(e_n,\spa\left(\left\{e_i\,|\,i\in\N,\ i\neq n\right\}\right)\right)\ge\varepsilon,\ n\in\N.
\end{equation}

Indeed, the opposite implies the existence of a subsequence $\left\{d_{n(k)}\right\}_{k=1}^\infty$ such that
\begin{equation*}
d_{n(k)}\to 0,\ k\to\infty.
\end{equation*}

Then, by selecting a vector
\[
f_{n(k)}\in 
\spa\left(\left\{e_i\,|\,i\in\N,\ i\neq n(k)\right\}\right),\ k\in\N,
\] 
such that 
\[
\|e_{n(k)}-f_{n(k)}\|<d_{n(k)}+1/k,\ k\in\N,
\]
we arrive at
\begin{multline*}
1=\|e_{n(k)}\|
\hfill
\text{since, by \eqref{ortho1}, 
$E_A(\Delta_{n(k)})f_{n(k)}=0$;}
\\
\shoveleft{
=\|E_A(\Delta_{n(k)})(e_{n(k)}-f_{n(k)})\|\
\le \|E_A(\Delta_{n(k)})\|\|e_{n(k)}-f_{n(k)}\|
\hfill
\text{by \eqref{bounded};}
}\\
\ \
\le M\|e_{n(k)}-f_{n(k)}\|\le M\left[d_{n(k)}+1/k\right]
\to 0,\ k\to\infty,
\hfill
\end{multline*}
which is a \textit{contradiction} proving \eqref{dist1}. 

As follows from the {\it Hahn-Banach Theorem}, for any $n\in\N$, there is an $e^*_n\in X^*$ such that 
\begin{equation}\label{H-B1}
\|e_n^*\|=1,\ n\in\N,\ \text{and}\ \langle e_i,e_j^*\rangle=\delta_{ij}d_i,\ i,j\in\N.
\end{equation}

Let us consider separately the two possibilities concerning the sequence of the real parts $\{\Rep\lambda_n\}_{n=1}^\infty$: its being \textit{bounded} or \textit{unbounded}. 

First, suppose that the sequence $\{\Rep\lambda_n\}_{n=1}^\infty$ is \textit{bounded}, i.e., there exists an $\omega>0$ such that
\begin{equation}\label{bounded1}
|\Rep\lambda_n| \le \omega,\ n\in\N,
\end{equation}
and consider the element
\begin{equation*}
f:=\sum_{k=1}^\infty k^{-2}e_k\in X,
\end{equation*}
which is well defined since $\left\{k^{-2}\right\}_{k=1}^\infty\in l_1$ ($l_1$ is the space of absolutely summable sequences) and $\|e_k\|=1$, $k\in\N$ (see \eqref{ortho1}).

In view of \eqref{ortho1}, by the properties of the \textit{s.m.},
\begin{equation}\label{vectors1}
E_A(\cup_{k=1}^\infty\Delta_k)f=f\ \text{and}\ E_A(\Delta_k)f=k^{-2}e_k,\ k\in\N.
\end{equation}

For an arbitrary $t\ge 0$ and any $g^*\in X^*$,
\begin{multline}\label{first1}
\int\limits_{\sigma(A)}e^{t\Rep\lambda}\,dv(f,g^*,\lambda)
\hfill \text{by \eqref{vectors1};}
\\
\shoveleft{
=\int\limits_{\sigma(A)} e^{t\Rep\lambda}\,d v(E_A(\cup_{k=1}^\infty \Delta_k)f,g^*,\lambda)
\hfill
\text{by \eqref{decompose};}
}\\
\shoveleft{
=\sum_{k=1}^\infty\int\limits_{\sigma(A)\cap\Delta_k}e^{t\Rep\lambda}\,dv(E_A(\Delta_k)f,g^*,\lambda)
\hfill 
\text{by \eqref{vectors1};}
}\\
\shoveleft{
=\sum_{k=1}^\infty k^{-2}\int\limits_{\sigma(A)\cap\Delta_k}e^{t\Rep\lambda}\,dv(e_k,g^*,\lambda)
}\\
\hfill
\text{since, for $\lambda\in \Delta_k$, by \eqref{bounded1} and \eqref{radii1},}\ 
\Rep\lambda=\Rep\lambda_k+(\Rep\lambda-\Rep\lambda_k)
\\
\hfill
\le \Rep\lambda_k+|\lambda-\lambda_k|\le \omega+\varepsilon_k\le \omega+1;
\\
\shoveleft{
\le e^{t(\omega+1)}\sum_{k=1}^\infty k^{-2}\int\limits_{\sigma(A)\cap\Delta_k}1\,dv(e_k,g^*,\lambda)
= e^{t(\omega+1)}\sum_{k=1}^\infty k^{-2}v(e_k,g^*,\Delta_k)
}\\
\hfill
\text{by \eqref{tv};}
\\
\hspace{1.2cm}
\le e^{t(\omega+1)}\sum_{k=1}^\infty k^{-2}4M\|e_k\|\|g^*\|
= 4Me^{t(\omega+1)}\|g^*\|\sum_{k=1}^\infty k^{-2}<\infty.
\hfill
\end{multline} 

Similarly, for an arbitrary $t\ge 0$ and any $n\in\N$,
\begin{multline}\label{second1}
\sup_{\{g^*\in X^*\,|\,\|g^*\|=1\}}
\int\limits_{\left\{\lambda\in\sigma(A)\,\middle|\,e^{t\Rep\lambda}>n\right\}} 
e^{t\Rep\lambda}\,dv(f,g^*,\lambda)
\\
\shoveleft{
\le 
\sup_{\{g^*\in X^*\,|\,\|g^*\|=1\}}e^{t(\omega+1)}\sum_{k=1}^\infty k^{-2}
\int\limits_{\left\{\lambda\in\sigma(A)\,\middle|\,e^{t\Rep\lambda}>n\right\}\cap \Delta_k}1\,dv(e_k,g^*,\lambda) 
}\\
\hfill \text{by \eqref{vectors1};}
\\
\shoveleft{
=e^{t(\omega+1)}\sup_{\{g^*\in X^*\,|\,\|g^*\|=1\}}\sum_{k=1}^\infty 
\int\limits_{\left\{\lambda\in\sigma(A)\,\middle|\,e^{t\Rep\lambda}>n\right\}\cap \Delta_k}1\,dv(E_A(\Delta_k)f,g^*,\lambda) 
}\\
\hfill \text{by \eqref{decompose};}
\\
\shoveleft{
= e^{t(\omega+1)}\sup_{\{g^*\in X^*\,|\,\|g^*\|=1\}}
\int\limits_{\{\lambda\in\sigma(A)\,|\,e^{t\Rep\lambda}>n\}}1\,dv(E_A(\cup_{k=1}^\infty\Delta_k)f,g^*,\lambda)
}\\
\hfill \text{by \eqref{vectors1};}
\\
\shoveleft{
= e^{t(\omega+1)}\sup_{\{g^*\in X^*\,|\,\|g^*\|=1\}}
\int\limits_{\{\lambda\in\sigma(A)\,|\,e^{t\Rep\lambda}>n\}}1\,dv(f,g^*,\lambda)
\hfill
\text{by \eqref{cond(ii)};}
}\\
\shoveleft{
\le e^{t(\omega+1)}\sup_{\{g^*\in X^*\,|\,\|g^*\|=1\}}4M\left\|E_A\left(\left\{\lambda\in\sigma(A)\,\middle|\,e^{t\Rep\lambda}>n\right\}\right)f\right\|\|g^*\|
}\\
\shoveleft{
\le 4Me^{t(\omega+1)}\left\|E_A\left(\left\{\lambda\in\sigma(A)\,\middle|\,e^{t\Rep\lambda}>n\right\}\right)f\right\|
}\\
\hfill
\text{by the strong continuity of the {\it s.m.};}
\\
\hspace{1.2cm}
\to 4Me^{t(\omega+1)}\left\|E_A\left(\emptyset\right)f\right\|=0,\ n\to\infty.
\hfill
\end{multline}

By Proposition \ref{prop}, \eqref{first1} and \eqref{second1} jointly imply that 
\[
f\in \bigcap\limits_{t\ge 0}D(e^{tA}),
\]
and hence, by Theorem \ref{GWS},
\[
y(t):=e^{tA}f,\ t\ge 0,
\]
is a weak solution of equation \eqref{1}.

Let
\begin{equation}\label{functional1}
h^*:=\sum_{k=1}^\infty k^{-2}e_k^*\in X^*,
\end{equation}
the functional being well defined since $\{k^{-2}\}_{k=1}^\infty\in l_1$ and $\|e_k^*\|=1$, $k\in\N$ (see \eqref{H-B1}).

In view of \eqref{H-B1} and \eqref{dist1}, we have:
\begin{equation}\label{funct-dist1}
\langle e_k,h^*\rangle=\langle e_k,k^{-2}e_k^*\rangle=d_k k^{-2}\ge \varepsilon k^{-2},\ k\in\N.
\end{equation}

For any $s>0$,
\begin{multline}\label{notin1}
\int\limits_{\sigma(A)}e^{s|\lambda|^{1/\beta}}e^{\Rep\lambda}\,dv(f,h^*,\lambda)
\hfill
\text{by \eqref{decompose} as in \eqref{first1};}
\\
\shoveleft{
=\sum_{k=1}^\infty k^{-2}\int\limits_{\sigma(A)\cap \Delta_k}e^{s|\lambda|^{1/\beta}}e^{\Rep\lambda}\,dv(e_k,h^*,\lambda)
}\\
\hfill
\text{since, for $\lambda\in \Delta_k$, by
\eqref{disks1}, \eqref{bounded1}, and \eqref{radii1}, 
$|\lambda|\ge k$ and}\ \Rep\lambda=
\\
\hfill
\Rep\lambda_k-(\Rep\lambda_k-\Rep\lambda)
\ge \Rep\lambda_k-|\Rep\lambda_k-\Rep\lambda|\ge -\omega-\varepsilon_k\ge -\omega-1;
\\
\shoveleft{
\ge \sum_{k=1}^\infty k^{-2}e^{sk^{1/\beta}}e^{-(\omega+1)} v(e_k,h^*,\Delta_k)\ge\sum_{k=1}^\infty e^{-(\omega+1)}k^{-2}e^{sk^{1/\beta}}|\langle E_A(\Delta_k)e_k,h^*\rangle|
}\\
\hfill
\text{by \eqref{ortho1} and \eqref{funct-dist1};}
\\
\hspace{1.2cm}
\ge \sum_{k=1}^\infty \varepsilon e^{-(\omega+1)} k^{-4}e^{sk^{1/\beta}}=\infty.
\hfill
\end{multline} 

By Proposition \ref{prop} and the properties of the \textit{o.c.} (see {\cite[Theorem XVIII.$2.11$ (f)]{Dun-SchIII}}), \eqref{notin1} implies that, for any $s>0$,
\[
f\notin D(e^{s|A|^{1/\beta}}e^{A}),
\]
and hence, in view of \eqref{GC},
\begin{equation*}
y(1)=e^{A}f\notin \bigcup_{s>0} D(e^{s|A|^{1/\beta}})
={\mathscr E}^{\{\beta\}}(A).
\end{equation*}

By Proposition \ref{particular}, we infer that the weak solution $y(t)=e^{tA}f$, $t\ge 0$, 
of equation \eqref{1} does not belong to the Roumieu
type Gevrey class ${\mathscr E}^{\{\beta \}}\left( (0,\infty),X\right)$, which completes our consideration of the case of the sequence's $\{\Rep\lambda_n\}_{n=1}^\infty$ being \textit{bounded}. 

Now, suppose that the sequence $\{\Rep\lambda_n\}_{n=1}^\infty$
is \textit{unbounded}. 

Therefore, there is a subsequence $\{\Rep\lambda_{n(k)}\}_{k=1}^\infty$ such that
\[
\Rep\lambda_{n(k)}\to \infty \ \text{or}\ \Rep\lambda_{n(k)}\to -\infty,\ k\to \infty.
\]
Let us consider separately each of the two cases.

First, suppose that 
\[
\Rep\lambda_{n(k)}\to \infty,\ k\to \infty.
\] 
Then, without loss of generality, we can regard that
\begin{equation}\label{infinity}
\Rep\lambda_{n(k)} \ge k,\ k\in\N.
\end{equation}

Consider the elements
\begin{equation*}
f:=\sum_{k=1}^\infty e^{-n(k)\Rep\lambda_{n(k)}}e_{n(k)}\in X
\ \text{and}\ h:=\sum_{k=1}^\infty e^{-\frac{n(k)}{2}\Rep\lambda_{n(k)}}e_{n(k)}\in X,
\end{equation*}
well defined since, by \eqref{infinity},
\[
\left\{e^{-n(k)\Rep\lambda_{n(k)}}\right\}_{k=1}^\infty,
\left\{e^{-\frac{n(k)}{2}\Rep\lambda_{n(k)}}\right\}_{k=1}^\infty
\in l_1
\]
and $\|e_{n(k)}\|=1$, $k\in\N$ (see \eqref{ortho1}).

By \eqref{ortho1},
\begin{equation}\label{subvectors1}
E_A(\cup_{k=1}^\infty\Delta_{n(k)})f=f\ \text{and}\
E_A(\Delta_{n(k)})f=e^{-n(k)\Rep\lambda_{n(k)}}e_{n(k)},\
k\in\N,
\end{equation}
and
\begin{equation}\label{subvectors12}
E_A(\cup_{k=1}^\infty\Delta_{n(k)})h=h\ \text{and}\
E_A(\Delta_{n(k)})h=e^{-\frac{n(k)}{2}\Rep\lambda_{n(k)}}e_{n(k)},\ k\in\N.
\end{equation}

For an arbitrary $t\ge 0$ and any $g^*\in X^*$, 
\begin{multline}\label{first2}
\int\limits_{\sigma(A)}e^{t\Rep\lambda}\,dv(f,g^*,\lambda)
\hfill
\text{by \eqref{decompose} as in \eqref{first1};}
\\
\shoveleft{
=\sum_{k=1}^\infty e^{-n(k)\Rep\lambda_{n(k)}}\int\limits_{\sigma(A)\cap\Delta_{n(k)}}e^{t\Rep\lambda}\,dv(e_{n(k)},g^*,\lambda)
}\\
\hfill
\text{since, for $\lambda\in \Delta_{n(k)}$, by \eqref{radii1},}\ \Rep\lambda
=\Rep\lambda_{n(k)}+(\Rep\lambda-\Rep\lambda_{n(k)})
\\
\hfill
\le \Rep\lambda_{n(k)}+|\lambda-\lambda_{n(k)}|\le \Rep\lambda_{n(k)}+1;
\\
\shoveleft{
\le \sum_{k=1}^\infty e^{-n(k)\Rep\lambda_{n(k)}}
e^{t(\Rep\lambda_{n(k)}+1)}
\int\limits_{\sigma(A)\cap\Delta_{n(k)}}1\,dv(e_{n(k)},g^*,\lambda)
}\\
\shoveleft{
= e^t\sum_{k=1}^\infty e^{-[n(k)-t]\Rep\lambda_{n(k)}}v(e_{n(k)},g^*,\Delta_{n(k)})
\hfill
\text{by \eqref{tv};}
}\\
\shoveleft{
\le e^t\sum_{k=1}^\infty e^{-[n(k)-t]\Rep\lambda_{n(k)}}4M\|e_{n(k)}\|\|g^*\|
= 4Me^t\|g^*\|\sum_{k=1}^\infty e^{-[n(k)-t]\Rep\lambda_{n(k)}}
}\\
\hspace{1.2cm}
<\infty.
\hfill
\end{multline}

Indeed, for all $k\in \N$ sufficiently large so that
\[
n(k)\ge t+1,
\]
in view of \eqref{infinity}, 
\[
e^{-[n(k)-t]\Rep\lambda_{n(k)}}\le e^{-k}.
\]

Similarly, for an arbitrary $t\ge 0$ and any $n\in\N$,
\begin{multline}\label{second2}
\sup_{\{g^*\in X^*\,|\,\|g^*\|=1\}}
\int\limits_{\left\{\lambda\in\sigma(A)\,\middle|\,e^{t\Rep\lambda}>n\right\}}e^{t\Rep\lambda}\,dv(f,g^*,\lambda)
\\
\shoveleft{
\le \sup_{\{g^*\in X^*\,|\,\|g^*\|=1\}}e^t\sum_{k=1}^\infty e^{-[n(k)-t]\Rep\lambda_{n(k)}}
\int\limits_{\left\{\lambda\in\sigma(A)\,\middle|\,e^{t\Rep\lambda}>n\right\}\cap \Delta_{n(k)}}1\,dv(e_{n(k)},g^*,\lambda)
}\\
\shoveleft{
=e^t\sup_{\{g^*\in X^*\,|\,\|g^*\|=1\}}\sum_{k=1}^\infty e^{-\left[\frac{n(k)}{2}-t\right]\Rep\lambda_{n(k)}}
e^{-\frac{n(k)}{2}\Rep\lambda_{(k)}}
}\\
\shoveleft{
\int\limits_{\left\{\lambda\in\sigma(A)\,\middle|\,e^{t\Rep\lambda}>n\right\}\cap \Delta_{n(k)}}1\,dv(e_{n(k)},g^*,\lambda)
}\\
\hfill
\text{since, by \eqref{infinity}, there is an $L>0$ such that
$e^{-\left[\frac{n(k)}{2}-t\right]\Rep\lambda_{n(k)}}\le L$, $k\in\N$;}
\\
\shoveleft{
\le Le^t\sup_{\{g^*\in X^*\,|\,\|g^*\|=1\}}\sum_{k=1}^\infty e^{-\frac{n(k)}{2}\Rep\lambda_{n(k)}}
\int\limits_{\left\{\lambda\in\sigma(A)\,\middle|\,e^{t\Rep\lambda}>n\right\}\cap \Delta_{n(k)}}1\,dv(e_{n(k)},g^*,\lambda)
}\\
\hfill
\text{by \eqref{subvectors12};}
\\
\shoveleft{
= Le^t\sup_{\{g^*\in X^*\,|\,\|g^*\|=1\}}\sum_{k=1}^\infty
\int\limits_{\left\{\lambda\in\sigma(A)\,\middle|\,e^{t\Rep\lambda}>n\right\}\cap \Delta_{n(k)}}1\,dv(E_A(\Delta_{n(k)})h,g^*,\lambda)
}\\
\hfill
\text{by \eqref{decompose};}
\\
\shoveleft{
= Le^t\sup_{\{g^*\in X^*\,|\,\|g^*\|=1\}}
\int\limits_{\left\{\lambda\in\sigma(A)\,\middle|\,e^{t\Rep\lambda}>n\right\}}1\,dv(E_A(\cup_{k=1}^\infty\Delta_{n(k)})h,g^*,\lambda)
}\\
\hfill
\text{by \eqref{subvectors12};}
\\
\shoveleft{
=Le^t\sup_{\{g^*\in X^*\,|\,\|g^*\|=1\}}\int\limits_{\{\lambda\in\sigma(A)\,|\,e^{t\Rep\lambda}>n\}}1\,dv(h,g^*,\lambda)
\hfill
\text{by \eqref{cond(ii)};}
}\\
\shoveleft{
\le Le^t\sup_{\{g^*\in X^*\,|\,\|g^*\|=1\}}4M
\left\|E_A\left(\left\{\lambda\in\sigma(A)\,\middle|\,e^{t\Rep\lambda}>n\right\}\right)h\right\|\|g^*\|
}\\
\shoveleft{
\le 4LMe^t\|E_A(\{\lambda\in\sigma(A)\,|\,e^{t\Rep\lambda}>n\})h\|
}\\
\hfill
\text{by the strong continuity of the {\it s.m.};}
\\
\hspace{1.2cm}
\to 4LMe^t\left\|E_A\left(\emptyset\right)h\right\|=0,\ n\to\infty.
\hfill
\end{multline}

By Proposition \ref{prop}, \eqref{first2} and \eqref{second2} jointly imply that 
\[
f\in \bigcap\limits_{t\ge 0}D(e^{tA}),
\]
and hence, by Theorem \ref{GWS},
\[
y(t):=e^{tA}f,\ t\ge 0,
\]
is a weak solution of equation \eqref{1}.

Since, for any $\lambda \in \Delta_{n(k)}$, $k\in \N$, by \eqref{radii1}, \eqref{infinity},
\begin{multline*}
\Rep\lambda =\Rep\lambda_{n(k)}-(\Rep\lambda_{n(k)}-\Rep\lambda)
\ge
\Rep\lambda_{n(k)}-|\Rep\lambda_{n(k)}-\Rep\lambda|
\\
\ \ \
\ge 
\Rep\lambda_{n(k)}-\varepsilon_{n(k)}
\ge \Rep\lambda_{n(k)}-1/n(k)\ge k-1\ge 0
\hfill
\end{multline*}
and, by \eqref{disks1},
\[
\Rep\lambda<n(k)^{-2}|\Imp\lambda|^{1/\beta},
\]
we infer that, for any $\lambda \in \Delta_{n(k)}$, $k\in \N$,
\begin{equation*}
|\lambda|\ge|\Imp\lambda|\ge 
\left[n(k)^2\Rep\lambda\right]^\beta\ge \left[n(k)^2(\Rep\lambda_{n(k)}-1/n(k))\right]^\beta.
\end{equation*}

Using this estimate, for an arbitrary $s>0$ and the functional $h^*\in X^*$ defined by \eqref{functional1}, we have:
\begin{multline}\label{notin}
\int\limits_{\sigma(A)}e^{s|\lambda|^{1/\beta}}\,dv(f,h^*,\lambda)
\hfill
\text{by \eqref{decompose} as in \eqref{first1};}
\\
\shoveleft{
=\sum_{k=1}^\infty e^{-n(k)\Rep\lambda_{n(k)}}\int\limits_{\sigma(A)\cap\Delta_{n(k)}}e^{s|\lambda|^{1/\beta}}\,dv(e_{n(k)},h^*,\lambda)
}\\
\shoveleft{
\ge\sum_{k=1}^\infty e^{-n(k)\Rep\lambda_{n(k)}}e^{sn(k)^2(\Rep\lambda_{n(k)}-1/n(k))}v(e_{n(k)},h^*,\Delta_{n(k)})
}\\
\shoveleft{
\ge \sum_{k=1}^\infty e^{-n(k)\Rep\lambda_{n(k)}}e^{sn(k)^2(\Rep\lambda_{n(k)}-1/n(k))}|\langle E_A(\Delta_{n(k)})e_{n(k)},h^*\rangle|
}\\
\hfill
\text{by \eqref{ortho1} and \eqref{funct-dist1};}
\\
\hspace{1.2cm}
\ge \sum_{k=1}^\infty \varepsilon
e^{(sn(k)-1)n(k)\Rep\lambda_{n(k)}-sn(k)}n(k)^{-2}
=\infty.
\hfill
\end{multline} 

Indeed, for all $k\in\N$ sufficiently large so that 
\begin{equation*}
sn(k)\ge s+2,
\end{equation*}
in view of \eqref{infinity},
\begin{multline*}
e^{(sn(k)-1)n(k)\Rep\lambda_{n(k)}-sn(k)}n(k)^{-2}
\ge 
e^{(s+1)n(k)-sn(k)}n(k)^{-2}=e^{n(k)}n(k)^{-2}
\\
\ \
\to\infty,\ k\to\infty.
\hfill
\end{multline*}

By Proposition \ref{prop} and the properties of the \textit{o.c.} (see {\cite[Theorem XVIII.$2.11$ (f)]{Dun-SchIII}}), \eqref{notin} implies that, for any $s>0$,
\[
f\notin D(e^{s|A|^{1/\beta}}e^{A}),
\]
which, in view of \eqref{GC}, further implies that
\begin{equation*}
y(1)=e^{A}f\notin \bigcup_{s>0} D(e^{s|A|^{1/\beta}})
={\mathscr E}^{\{\beta\}}(A).
\end{equation*}

Whence, by Proposition \ref{particular}, we infer that the weak solution $y(t)=e^{tA}f$, $t\ge 0$, 
of equation \eqref{1} does not belong to the Roumieu
type Gevrey class ${\mathscr E}^{\{\beta \}}\left( (0,\infty),X\right)$. 

Now, suppose that 
\[
\Rep\lambda_{n(k)}\to -\infty,\ k\to \infty
\] 
Then, without loss of generality, we can regard that
\begin{equation}\label{infinity2}
\Rep\lambda_{n(k)} \le -k,\ k\in\N.
\end{equation}

Consider the element
\begin{equation*}
f:=\sum_{k=1}^\infty k^{-2}e_{n(k)}\in X,
\end{equation*}
which is well defined since $\{k^{-2}\}_{k=1}^\infty\in l_1$ and $\|e_{n(k)}\|=1$, $k\in\N$ (see \eqref{ortho1}).

By \eqref{ortho1},
\begin{equation}\label{subvectors2}
E_A(\cup_{k=1}^\infty\Delta_{n(k)})f=f\ \text{and}\
E_A(\Delta_{n(k)})f=k^{-2}e_{n(k)},\
k\in\N.
\end{equation}

For an arbitrary $t\ge 0$ and any $g^*\in X^*$,
\begin{multline}\label{first3}
\int\limits_{\sigma(A)}e^{t\Rep\lambda}\,dv(f,g^*,\lambda)
\hfill
\text{by \eqref{decompose} as in \eqref{first1};}
\\
\shoveleft{
=\sum_{k=1}^\infty k^{-2}\int\limits_{\sigma(A)\cap\Delta_{n(k)}}e^{t\Rep\lambda}\,dv(e_{n(k)},g^*,\lambda)
}\\
\hfill
\text{since, for $\lambda\in \Delta_{n(k)}$, by
\eqref{infinity2} and \eqref{radii1},}
\\
\hfill
\Rep\lambda=\Rep\lambda_{n(k)}+(\Rep\lambda-\Rep\lambda_{n(k)})\le \Rep\lambda_{n(k)}+|\lambda-\lambda_{n(k)}|\le -k+1\le 0;
\\
\shoveleft{
\le 
\sum_{k=1}^\infty k^{-2}\int\limits_{\sigma(A)\cap\Delta_{n(k)}}1\,dv(e_{n(k)},g^*,\lambda)
=
\sum_{k=1}^\infty k^{-2}v(e_{n(k)},g^*,\Delta_{n(k)})
\hfill
\text{by \eqref{tv};}
}\\
\hspace{1.2cm}
\le \sum_{k=1}^\infty k^{-2}4M\|e_{n(k)}\|\|g^*\|
= 4M\|g^*\|\sum_{k=1}^\infty k^{-2}
<\infty.
\hfill
\end{multline}

Similarly, for an arbitrary $t\ge 0$ and any $n\in\N$,
\begin{multline}\label{second3}
\sup_{\{g^*\in X^*\,|\,\|g^*\|=1\}}
\int\limits_{\left\{\lambda\in\sigma(A)\,\middle|\,e^{t\Rep\lambda}>n\right\}}e^{t\Rep\lambda}\,dv(f,g^*,\lambda)
\\
\shoveleft{
\le \sup_{\{g^*\in X^*\,|\,\|g^*\|=1\}}\sum_{k=1}^\infty k^{-2}
\int\limits_{\left\{\lambda\in\sigma(A)\,\middle|\,e^{t\Rep\lambda}>n\right\}\cap \Delta_{n(k)}}1\,dv(e_{n(k)},g^*,\lambda)
}\\
\hfill
\text{by \eqref{subvectors2};}
\\
\shoveleft{
\le \sup_{\{g^*\in X^*\,|\,\|g^*\|=1\}}\sum_{k=1}^\infty 
\int\limits_{\left\{\lambda\in\sigma(A)\,\middle|\,e^{t\Rep\lambda}>n\right\}\cap \Delta_{n(k)}}1\,dv(E_A(\Delta_{n(k)})f,g^*,\lambda)
}\\
\hfill
\text{by \eqref{decompose};}
\\
\shoveleft{
=\sup_{\{g^*\in X^*\,|\,\|g^*\|=1\}}\int\limits_{\{\lambda\in\sigma(A)\,|\,e^{t\Rep\lambda}>n\}}1\,dv(E_A(\cup_{k=1}^\infty\Delta_{n(k)})f,g^*,\lambda)
\hfill
\text{by \eqref{subvectors2};}
}\\
\shoveleft{
=\sup_{\{g^*\in X^*\,|\,\|g^*\|=1\}}\int\limits_{\{\lambda\in\sigma(A)\,|\,e^{t\Rep\lambda}>n\}}1\,dv(f,g^*,\lambda)
\hfill
\text{by \eqref{cond(ii)};}
}\\
\shoveleft{
\le \sup_{\{g^*\in X^*\,|\,\|g^*\|=1\}}4M
\left\|E_A\left(\left\{\lambda\in\sigma(A)\,\middle|\,e^{t\Rep\lambda}>n\right\}\right)f\right\|\|g^*\|
}\\
\shoveleft{
\le 4M\|E_A(\{\lambda\in\sigma(A)\,|\,e^{t\Rep\lambda}>n\})f\|
}\\
\hfill
\text{by the strong continuity of the {\it s.m.};}
\\
\hspace{1.2cm}
\to 4M\left\|E_A\left(\emptyset\right)f\right\|=0,\ n\to\infty.
\hfill
\end{multline}

By Proposition \ref{prop}, \eqref{first3} and \eqref{second3} jointly imply that 
\[
f\in \bigcap\limits_{t\ge 0}D(e^{tA}),
\]
and hence,
by Theorem \ref{GWS},
\[
y(t):=e^{tA}f,\ t\ge 0,
\]
is a weak solution of equation \eqref{1}.

Let
\begin{equation}\label{functional2}
h^*:=\sum_{k=1}^\infty k^{-2}e_{n(k)}^*\in X^*,
\end{equation}
the functional being well defined since $\{k^{-2}\}_{k=1}^\infty\in l_1$ and $\|e_{n(k)}^*\|=1$, $k\in\N$ (see \eqref{H-B1}).

In view of \eqref{H-B1} and \eqref{dist1}, we have:
\begin{equation}\label{ffunct-dist2}
\langle e_{n(k)},h^*\rangle=\langle e_{n(k)},k^{-2}e_{n(k)}^*\rangle=d_{n(k)}k^{-2}\ge \varepsilon k^{-2},\ k\in\N.
\end{equation}

Since, for any $\lambda \in \Delta_{n(k)}$, $k\in \N$, by \eqref{infinity2} and \eqref{radii1},
\begin{multline}\label{-inf}
\Rep\lambda =\Rep\lambda_{n(k)}+(\Rep\lambda-\Rep\lambda_{n(k)})
\le
\Rep\lambda_{n(k)}+|\Rep\lambda-\Rep\lambda_{n(k)}|
\\
\hspace{1.2cm}
\le 
\Rep\lambda_{n(k)}+\varepsilon_{n(k)}\le -k+1 \le 0
\hfill
\end{multline}
and, by \eqref{disks1},
\[
-n(k)^{-1}|\Imp\lambda|^{1/\beta}<\Rep\lambda,
\]
we infer that, for any $\lambda \in \Delta_{n(k)}$, $k\in \N$,
\begin{equation*}
|\lambda|\ge|\Imp\lambda|\ge 
\left[n(k)(-\Rep\lambda)\right]^\beta.
\end{equation*}

Using this estimate, for an arbitrary $s>0$ and the functional $h^*\in X^*$ defined by \eqref{functional2}, we have:
\begin{multline}\label{notin3}
\int\limits_{\sigma(A)}e^{s|\lambda|^{1/\beta}}e^{\Rep\lambda}\,dv(f,h^*,\lambda)
\hfill
\text{by \eqref{decompose} as in \eqref{first1};}
\\
\shoveleft{
=\sum_{k=1}^\infty k^{-2}\int\limits_{\sigma(A)\cap\Delta_{n(k)}}e^{s|\lambda|^{1/\beta}}e^{\Rep\lambda}\,dv(e_{n(k)},h^*,\lambda)
}\\
\hspace{1.2cm}
\ge\sum_{k=1}^\infty k^{-2}\int\limits_{\sigma(A)\cap\Delta_{n(k)}}e^{[sn(k)-1](-\Rep \lambda)}\,dv(e_{n(k)},h^*,\lambda)=\infty.
\hfill
\end{multline} 

Indeed, for all $k\in\N$ sufficiently large so that 
\begin{equation*}
sn(k)\ge 2,
\end{equation*}
we have:
\begin{multline*}
k^{-2}\int\limits_{\sigma(A)\cap\Delta_{n(k)}}e^{[sn(k)-1](-\Rep \lambda)}\,dv(e_{n(k)},h^*,\lambda)
\ge k^{-2}\int\limits_{\sigma(A)\cap\Delta_{n(k)}}e^{-\Rep \lambda}\,dv(e_{n(k)},h^*,\lambda)
\\
\hfill
\text{by \eqref{-inf};}
\\
\shoveleft{
\ge k^{-2}e^{k-1}\int\limits_{\sigma(A)\cap\Delta_{n(k)}}1\,dv(e_{n(k)},h^*,\lambda)
=k^{-2}e^{k-1}
v(e_{n(k)},h^*,\Delta_{n(k)})
}\\
\shoveleft{
\ge k^{-2}e^{k-1}|\langle E_A(\Delta_{n(k)})e_{n(k)},h^*\rangle|
\hfill
\text{by \eqref{ortho1} and \eqref{ffunct-dist2};}
}\\
\ \
\ge \varepsilon k^{-4}e^{k-1}\to\infty,\ k\to\infty.
\hfill
\end{multline*} 

By Proposition \ref{prop} and the properties of the \textit{o.c.} (see {\cite[Theorem XVIII.$2.11$ (f)]{Dun-SchIII}}), \eqref{notin3} implies that, for any $s>0$,
\[
f\notin D(e^{s|A|^{1/\beta}}e^{A}),
\]
which, in view of \eqref{GC}, further implies that
\begin{equation*}
y(1)=e^{A}f\notin \bigcup_{s>0} D(e^{s|A|^{1/\beta}})
={\mathscr E}^{\{\beta\}}(A).
\end{equation*}

Whence, by Proposition \ref{particular}, we infer that the weak solution $y(t)=e^{tA}f$, $t\ge 0$, 
of equation \eqref{1} does not belong to the Roumieu
type Gevrey class ${\mathscr E}^{\{\beta \}}\left( (0,\infty),X\right)$, which completes our consideration of the case of
the sequence's $\{\Rep\lambda_n\}_{n=1}^\infty$ being \textit{unbounded}. 

With every possibility concerning $\{\Rep\lambda_n\}_{n=1}^\infty$ considered, 
the proof by contrapositive of the \textit{``only if" part} is complete and so is the proof of the 
entire statement.
\end{proof}

For $\beta=1$, we obtain the following important particular case.

\begin{cor}[Characterization of the Analyticity of Weak Solutions on $(0,\infty)$]\label{CAWS}
Let $A$ be a scalar type spectral operator in a complex Banach space $(X,\|\cdot\|)$. Every weak solution of the equation \eqref{1} is analytic on $(0,\infty)$ iff there exist $b_+>0$ and $ b_->0$ such that the set $\sigma(A)\setminus {\mathscr P}^1_{b_-,b_+}$,
where
\begin{equation*}
{\mathscr P}^1_{b_-,b_+}:=\left\{\lambda \in \C\, \middle|\,
\Rep\lambda \le -b_-|\Imp\lambda| 
\ \text{or}\ 
\Rep\lambda \ge b_+|\Imp\lambda|\right\},
\end{equation*}
is bounded 
(see Fig. \ref{fig:graph6}).

\begin{figure}[h]
\centering
\includegraphics[height=2in]{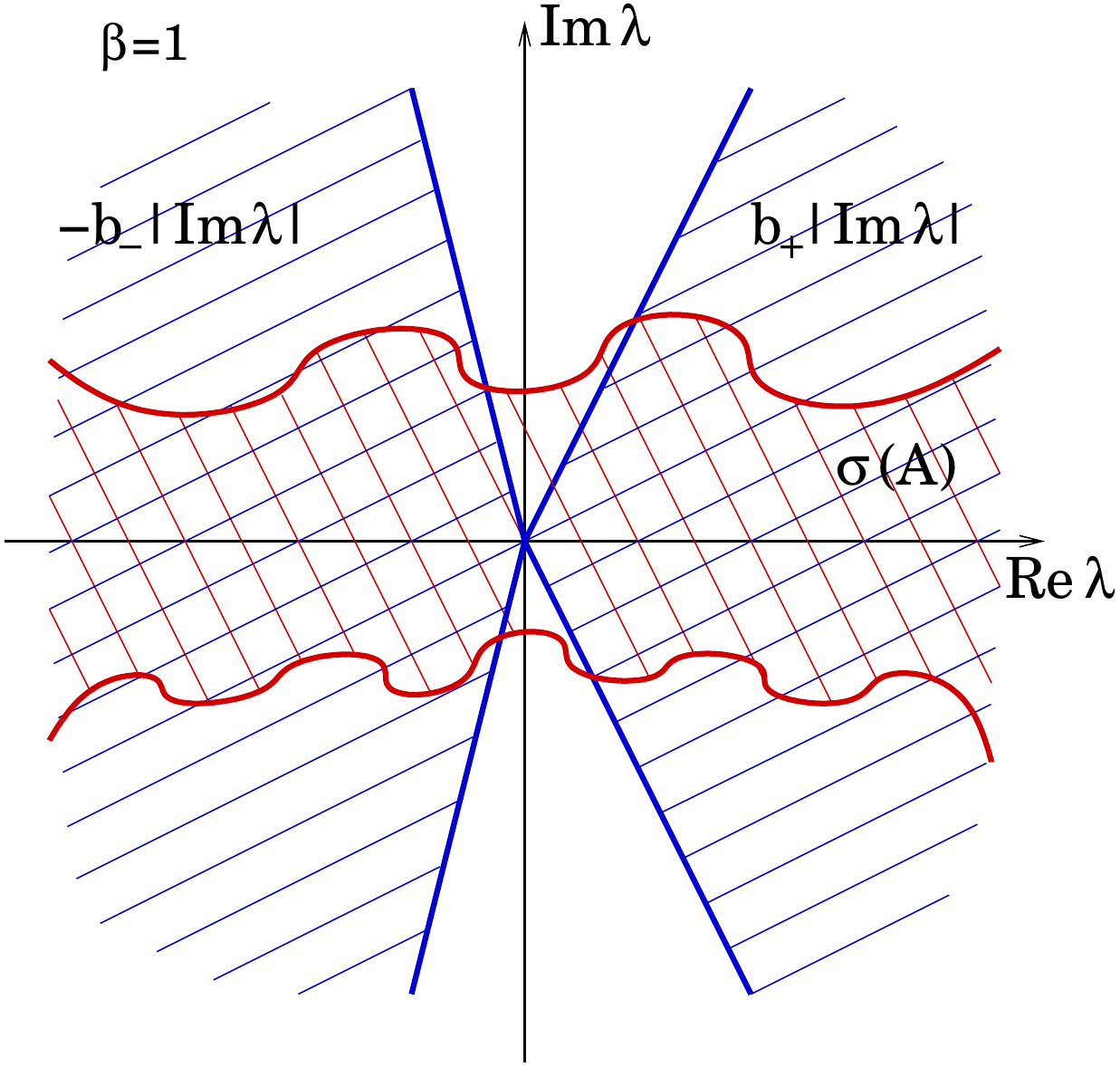}
\caption[]{}
\label{fig:graph6}
\end{figure}
\end{cor}

\begin{rem}
Thus, we have obtained a generalization of {\cite[Theorem $4.2$]{Markin2001(1)}}, the counterpart for a normal operator $A$ in a complex Hilbert space, and of {\cite[Theorem $5.1$]{Markin2004(1)}} (cf. \cite{Markin2008}), a characterization of the generation of a Roumieu type Gevrey ultradifferentiable $C_0$-semigroup by a scalar type spectral operator $A$.
\end{rem}

Now, let us treat the Beurling type strong Gevrey ultradifferentiability of order $\beta>1$. Observe that the case of \textit{entireness} ($\beta=1$) is included in {\cite[Theorem $4.1$]{Markin2018(4)}}
(see also {\cite[Corollary $4.1$]{Markin2018(4)}}).

\begin{thm}\label{open2}
Let $A$ be a scalar type spectral operator in a complex Banach space $(X,\|\cdot\|)$ with spectral measure $E_A(\cdot)$ and $ 1<\beta<\infty$.
Every weak solution of equation \eqref{1} belongs to the $\beta$th-order Beurling type Gevrey class 
${\mathscr E}^{(\beta)}\left((0,\infty),X\right)$ iff there exists a $b_+>0$ such that, for any $b_->0$, the set $\sigma(A)\setminus {\mathscr P}^\beta_{b_-,b_+}$,
where
\begin{equation*}
{\mathscr P}^\beta_{b_-,b_+}:=\left\{\lambda \in \C\, \middle|\,
\Rep\lambda \le -b_-|\Imp\lambda|^{1/\beta}\ \text{or}\ 
\Rep\lambda \ge b_+|\Imp\lambda|^{1/\beta} \right\},
\end{equation*}
is bounded (see Fig. \ref{fig:graph4}).
\end{thm}

\begin{proof}\

\textit{``If"} Part. Suppose that there exists a $b_+>0$ such that, for any $b_->0$, the set $\sigma(A)\setminus {\mathscr P}^\beta_{b_-,b_+}$
is \textit{bounded} and let $y(\cdot)$ be an arbitrary weak solution of equation \eqref{1}. 

By Theorem \ref{GWS}, 
\begin{equation*}
y(t)=e^{tA}f,\ t\ge 0,\ \text{with some}\
f \in \bigcap_{t\ge 0}D(e^{tA}).
\end{equation*}

Our purpose is to show that $y(\cdot)\in {\mathscr E}^{(\beta)}\left((0,\infty),X\right)$, which, by Proposition \ref{particular} and \eqref{GC}, is attained by showing that, for each $t>0$,
\[
y(t)\in {\mathscr E}^{(\beta)}\left(A\right)
=\bigcap_{s>0} D(e^{s|A|^{1/\beta}}).
\]

Let us proceed by proving that, for any $t>0$ and $s>0$,
\[
y(t)\in D(e^{s|A|^{1/\beta}})
\] 
via Proposition \ref{prop}.

Since $\beta>1$, for any $b_->0$,
there exists a $c(b_-)>0$ such that
\begin{equation}\label{est}
x\le b_-^{-\beta}x^{\beta},\ x\ge c(b_-).
\end{equation}

Fixing arbitrary $t>0$ and $s>0$, since $b_->0$ is random, we can set 
\begin{equation}\label{b_-}
b_-:=2^{1/\beta}st^{-1}>0,
\end{equation}
such a peculiar choice explaining itself in the process. 

For arbitrary $t>0$ and $s>0$, $b_->0$ chosen as in \eqref{b_-}, and any $g^*\in X^*$,
\begin{multline}\label{ffirst}
\int\limits_{\sigma(A)}e^{s|\lambda|^{1/\beta}}e^{t\Rep\lambda}\,dv(f,g^*,\lambda)
=\int\limits_{\sigma(A)\setminus{\mathscr P}_{b_-,b_+}^\beta}e^{s|\lambda|^{1/\beta}}e^{t\Rep\lambda}\,dv(f,g^*,\lambda)
\\
\shoveleft{
+\int\limits_{\left\{\lambda\in \sigma(A)\cap{\mathscr P}_{b_-,b_+}^\beta\,\middle|\,-c(b_-)<\Rep\lambda<1 \right\}}e^{s|\lambda|^{1/\beta}}e^{t\Rep\lambda}\,dv(f,g^*,\lambda)
}\\
\shoveleft{
+\int\limits_{\left\{\lambda\in \sigma(A)\cap{\mathscr P}_{b_-,b_+}^\beta\,\middle|\,\Rep\lambda\ge 1 \right\}}e^{s|\lambda|^{1/\beta}}e^{t\Rep\lambda}\,dv(f,g^*,\lambda)
}\\
\hspace{1.2cm}
+\int\limits_{\left\{\lambda\in \sigma(A)\cap{\mathscr P}_{b_-,b_+}^\beta\,\middle|\,\Rep\lambda\le -c(b_-) \right\}}e^{s|\lambda|^{1/\beta}}e^{t\Rep\lambda}\,dv(f,g^*,\lambda)<\infty.
\hfill
\end{multline}

Indeed, 
\[
\int\limits_{\sigma(A)\setminus{\mathscr P}_{b_-,b_+}^\beta}e^{s|\lambda|^{1/\beta}}e^{t\Rep\lambda}\,dv(f,g^*,\lambda)<\infty
\]
and
\[
\int\limits_{\left\{\lambda\in \sigma(A)\cap{\mathscr P}_{b_-,b_+}^\beta\,\middle|\,-c(b_-)<\Rep\lambda<1 \right\}}e^{s|\lambda|^{1/\beta}}e^{t\Rep\lambda}\,dv(f,g^*,\lambda)<\infty
\]
due to the boundedness of the sets
\[
\sigma(A)\setminus{\mathscr P}_{b_-,b_+}^\beta\ \text{and}\
\left\{\lambda\in \sigma(A)\cap{\mathscr P}_{b_-,b_+}^\beta\;\middle|\;-c(b_-)<\Rep\lambda<1 \right\},
\]
the continuity of the integrated function 
on $\C$, and the finiteness of the measure $v(f,g^*,\cdot)$.

Further, for arbitrary $t>0$, $s>0$, $b_->0$ chosen as in \eqref{b_-}, and any $g^*\in X^*$,
\begin{multline}\label{iinterm}
\int\limits_{\left\{\lambda\in \sigma(A)\cap{\mathscr P}_{b_-,b_+}^\beta\,\middle|\,\Rep\lambda\ge 1 \right\}}e^{s|\lambda|^{1/\beta}}e^{t\Rep\lambda}\,dv(f,g^*,\lambda)
\\
\shoveleft{
\le\int\limits_{\left\{\lambda\in \sigma(A)\cap{\mathscr P}_{b_-,b_+}^\beta\,\middle|\,\Rep\lambda\ge 1 \right\}}e^{s\left[|\Rep\lambda|+|\Imp\lambda|\right]^{1/\beta}}e^{t\Rep\lambda}\,dv(f,g^*,\lambda)
}\\
\hfill
\text{since, for $\lambda\in\sigma(A)\cap{\mathscr P}_{b_-,b_+}^\beta$ with $\Rep\lambda\ge 1$, $b_+^{-\beta}\Rep\lambda^\beta\ge |\Imp\lambda|$;}
\\
\shoveleft{
\le 
\int\limits_{\left\{\lambda\in \sigma(A)\cap{\mathscr P}_{b_-,b_+}^\beta\,\middle|\,\Rep\lambda\ge 1 \right\}}e^{s\left[\Rep\lambda+b_+^{-\beta}\Rep\lambda^\beta\right]^{1/\beta}}e^{t\Rep\lambda}\,dv(f,g^*,\lambda)
}\\
\hfill
\text{since, in view of $\Rep\lambda\ge 1$ and $\beta>1$, $\Rep\lambda^\beta\ge\Rep\lambda$;}
\\
\shoveleft{
\le 
\int\limits_{\left\{\lambda\in \sigma(A)\cap{\mathscr P}_{b_-,b_+}^\beta\,\middle|\,\Rep\lambda\ge 1 \right\}}e^{s\left(1+b_+^{-\beta}\right)^{1/\beta}\Rep\lambda}e^{t\Rep\lambda}\,dv(f,g^*,\lambda)
}\\
\shoveleft{
= \int\limits_{\left\{\lambda\in \sigma(A)\cap{\mathscr P}_{b_-,b_+}^\beta\,\middle|\,\Rep\lambda\ge 1 \right\}}e^{\left[s\left(1+b_+^{-\beta}\right)^{1/\beta}+t\right]\Rep\lambda}\,dv(f,g^*,\lambda)
}\\
\hfill
\text{since $f\in \bigcap\limits_{t\ge 0}D(e^{tA})$, by Proposition \ref{prop};}
\\
\hspace{1.2cm}
<\infty. 
\hfill
\end{multline}

Observe that, for the finiteness of the three preceding integrals, the choice of $b_->0$ is superfluous.

Finally, for arbitrary $t>0$ and $s>0$, $b_->0$ chosen as in \eqref{b_-}, and any $g^*\in X^*$,
\begin{multline}\label{iinterm2}
\int\limits_{\left\{\lambda\in \sigma(A)\cap{\mathscr P}_{b_-,b_+}^\beta\,\middle|\,\Rep\lambda\le -c(b_-) \right\}}e^{s|\lambda|^{1/\beta}}e^{t\Rep\lambda}\,dv(f,g^*,\lambda)
\\
\shoveleft{
\le\int\limits_{\left\{\lambda\in \sigma(A)\cap{\mathscr P}_{b_-,b_+}^\beta\,\middle|\,\Rep\lambda\le -c(b_-) \right\}}e^{s\left[|\Rep\lambda|+|\Imp\lambda|\right]^{1/\beta}}e^{t\Rep\lambda}\,dv(f,g^*,\lambda)
}\\
\hfill
\text{since, for $\lambda\in\sigma(A)\cap{\mathscr P}_{b_-,b_+}^\beta$ with $\Rep\lambda\le -c(b_-)$, $b_-^{-\beta}(-\Rep\lambda)^\beta\ge |\Imp\lambda|$;}
\\
\shoveleft{
\le 
\int\limits_{\left\{\lambda\in \sigma(A)\cap{\mathscr P}_{b_-,b_+}^\beta\,\middle|\,\Rep\lambda\le -c(b_-) \right\}}e^{s\left[-\Rep\lambda+b_-^{-\beta}(-\Rep\lambda)^\beta\right]^{1/\beta}}e^{t\Rep\lambda}\,dv(f,g^*,\lambda)
}\\
\hfill
\text{since, in view of $\Rep\lambda\le -c(b_-)$, by \eqref{est}, $b_-^{-\beta}(-\Rep\lambda)^\beta\ge-\Rep\lambda$;}
\\
\shoveleft{
\le 
\int\limits_{\left\{\lambda\in \sigma(A)\cap{\mathscr P}_{b_-,b_+}^\beta\,\middle|\,\Rep\lambda\le -c(b_-) \right\}}e^{s\left(2b_-^{-\beta}\right)^{1/\beta}(-\Rep\lambda)}e^{t\Rep\lambda}\,dv(f,g^*,\lambda)
}\\
\shoveleft{
= \int\limits_{\left\{\lambda\in \sigma(A)\cap{\mathscr P}_{b_-,b_+}^\beta\,\middle|\,\Rep\lambda\le -c(b_-) \right\}}e^{\left[t-s2^{1/\beta}b_-^{-1}\right]\Rep\lambda}\,dv(f,g^*,\lambda)
}\\
\hfill
\text{since $b_-:=2^{1/\beta}st^{-1}>0$ (see \eqref{b_-});}
\\
\shoveleft{
= \int\limits_{\left\{\lambda\in \sigma(A)\cap{\mathscr P}_{b_-,b_+}^\beta\,\middle|\,\Rep\lambda\le -c(b_-) \right\}}1\,dv(f,g^*,\lambda)
\le \int\limits_{\sigma(A)}1\,dv(f,g^*,\lambda)
}\\
\shoveleft{
=v(f,g^*,\sigma(A))
\hfill
\text{by the \eqref{tv};}
}\\
\hspace{1.2cm}
\le 4M\|f\|\|g^*\|<\infty. 
\hfill
\end{multline}

Also, for arbitrary $t>0$ and $s>0$, $b_->0$ chosen as in \eqref{b_-}, and any $n\in\N$,
\begin{multline}\label{ssecond}
\sup_{\{g^*\in X^*\,|\,\|g^*\|=1\}}
\int\limits_{\left\{\lambda\in\sigma(A)\,\middle|\,e^{s|\lambda|^{1/\beta}}e^{t\Rep\lambda}>n\right\}}
e^{s|\lambda|^{1/\beta}}e^{t\Rep\lambda}\,dv(f,g^*,\lambda)
\\
\shoveleft{
\le \sup_{\{g^*\in X^*\,|\,\|g^*\|=1\}}
\int\limits_{\left\{\lambda\in\sigma(A)\setminus{\mathscr P}_{b_-,b_+}^\beta\,\middle|\,e^{s|\lambda|^{1/\beta}}e^{t\Rep\lambda}>n\right\}}e^{s|\lambda|^{1/\beta}}e^{t\Rep\lambda}\,dv(f,g^*,\lambda)
}\\
\shoveleft{
+ \sup_{\{g^*\in X^*\,|\,\|g^*\|=1\}}
\int\limits_{\left\{\lambda\in\sigma(A)\cap{\mathscr P}_{b_-,b_+}^\beta\,\middle|\,-c(b_-)<\Rep\lambda<1,\, e^{s|\lambda|^{1/\beta}}e^{t\Rep\lambda}>n\right\}}e^{s|\lambda|^{1/\beta}}e^{t\Rep\lambda}\,dv(f,g^*,\lambda)
}\\
\shoveleft{
+ \sup_{\{g^*\in X^*\,|\,\|g^*\|=1\}}
\int\limits_{\left\{\lambda\in\sigma(A)\cap{\mathscr P}_{b_-,b_+}^\beta\,\middle|\,\Rep\lambda\ge 1,\, e^{s|\lambda|^{1/\beta}}e^{t\Rep\lambda}>n\right\}}e^{s|\lambda|^{1/\beta}}e^{t\Rep\lambda}\,dv(f,g^*,\lambda)
}\\
\shoveleft{
+ \sup_{\{g^*\in X^*\,|\,\|g^*\|=1\}}
\int\limits_{\left\{\lambda\in\sigma(A)\cap{\mathscr P}_{b_-,b_+}^\beta\,\middle|\,\Rep\lambda\le -c(b_-),\, e^{s|\lambda|^{1/\beta}}e^{t\Rep\lambda}>n\right\}}e^{s|\lambda|^{1/\beta}}e^{t\Rep\lambda}\,dv(f,g^*,\lambda)
}\\
\hspace{1.2cm}
\to 0,\ n\to\infty.
\hfill
\end{multline}

Indeed, since, due to the boundedness of the sets
\[
\sigma(A)\setminus{\mathscr P}_{b_-,b_+}^\beta\ \text{and}\
\left\{\lambda\in\sigma(A)\cap{\mathscr P}_{b_-,b_+}^\beta\,\middle|\,-c(b_-)<\Rep\lambda<1\right\}
\]
and the continuity of the integrated function on $\C$,
the sets
\[
\left\{\lambda\in\sigma(A)\setminus{\mathscr P}_{b_-,b_+}^\beta\,\middle|\,e^{s|\lambda|^{1/\beta}}e^{t\Rep\lambda}>n\right\}
\]
and 
\[
\left\{\lambda\in\sigma(A)\cap{\mathscr P}_{b_-,b_+}^\beta\,\middle|\,-c(b_-)<\Rep\lambda<1,\, e^{s|\lambda|^{1/\beta}}e^{t\Rep\lambda}>n\right\}
\]
are \textit{empty} for all sufficiently large $n\in \N$,
we immediately infer that, for any $t>0$, $s>0$ and $b_->0$ chosen as in \eqref{b_-},
\[
\lim_{n\to\infty}\sup_{\{g^*\in X^*\,|\,\|g^*\|=1\}}
\int\limits_{\left\{\lambda\in\sigma(A)\setminus{\mathscr P}_{b_-,b_+}^\beta\,\middle|\,e^{s|\lambda|^{1/\beta}}e^{t\Rep\lambda}>n\right\}}e^{s|\lambda|^{1/\beta}}e^{t\Rep\lambda}\,dv(f,g^*,\lambda)=0
\]
and
\[
\lim_{n\to\infty}\sup_{\{g^*\in X^*\,|\,\|g^*\|=1\}}
\int\limits_{\left\{\lambda\in\sigma(A)\cap{\mathscr P}_{b_-,b_+}^\beta\,\middle|\,-c(b_-)<\Rep\lambda<1,\, e^{s|\lambda|^{1/\beta}}e^{t\Rep\lambda}>n\right\}}e^{s|\lambda|^{1/\beta}}e^{t\Rep\lambda}\,dv(f,g^*,\lambda)
=0.
\]

Further, for arbitrary $t>0$, $s>0$, $b_->0$ chosen as in \eqref{b_-}, and any $n\in\N$,
\begin{multline*}
\sup_{\{g^*\in X^*\,|\,\|g^*\|=1\}}
\int\limits_{\left\{\lambda\in\sigma(A)\cap{\mathscr P}_{b_-,b_+}^\beta\,\middle|\,\Rep\lambda\ge 1,\, e^{s|\lambda|^{1/\beta}}e^{t\Rep\lambda}>n\right\}}e^{s|\lambda|^{1/\beta}}e^{t\Rep\lambda}\,dv(f,g^*,\lambda)
\\
\hfill
\text{as in \eqref{iinterm};}
\\
\shoveleft{
\le \sup_{\{g^*\in X^*\,|\,\|g^*\|=1\}}
\int\limits_{\left\{\lambda\in\sigma(A)\cap{\mathscr P}_{b_-,b_+}^\beta\,\middle|\,\Rep\lambda\ge 1,\, e^{s|\lambda|^{1/\beta}}e^{t\Rep\lambda}>n\right\}}e^{\left[s\left(1+b_+^{-\beta}\right)^{1/\beta}+t\right]\Rep\lambda}\,dv(f,g^*,\lambda)
}\\
\hfill
\text{since $f\in \bigcap\limits_{t\ge 0}D(e^{tA})$, by \eqref{cond(ii)};}
\\
\shoveleft{
\le \sup_{\{g^*\in X^*\,|\,\|g^*\|=1\}}
}\\
\shoveleft{
4M\left\|E_A\left(\left\{\lambda\in\sigma(A)\cap{\mathscr P}_{b_-,b_+}^\beta\,\middle|\,\Rep\lambda\ge 1,\, e^{s|\lambda|^{1/\beta}}e^{t\Rep\lambda}>n\right\}\right)
e^{\left[s\left(1+b_+^{-\beta}\right)^{1/\beta}+t\right]A}f\right\|\|g^*\|
}\\
\shoveleft{
\le 4M\left\|E_A\left(\left\{\lambda\in\sigma(A)\cap{\mathscr P}_{b_-,b_+}^\beta\,\middle|\,\Rep\lambda\ge 1,\, e^{s|\lambda|^{1/\beta}}e^{t\Rep\lambda}>n\right\}\right)
e^{\left[s\left(1+b_+^{-\beta}\right)^{1/\beta}+t\right]A}f\right\|
}\\
\hfill
\text{by the strong continuity of the {\it s.m.};}
\\
\ \
\to 4M\left\|E_A\left(\emptyset\right)e^{\left[s\left(1+b_+^{-\beta}\right)^{1/\beta}+t\right]A}f\right\|=0,\ n\to\infty.
\hfill
\end{multline*}

Finally, for arbitrary $t>0$ and $s>0$, $b_->0$ chosen as in \eqref{b_-}, and any $n\in\N$,
\begin{multline*}
\sup_{\{g^*\in X^*\,|\,\|g^*\|=1\}}
\int\limits_{\left\{\lambda\in\sigma(A)\cap{\mathscr P}_{b_-,b_+}^\beta\,\middle|\,\Rep\lambda\le -c(b_-),\, e^{s|\lambda|^{1/\beta}}e^{t\Rep\lambda}>n\right\}}e^{s|\lambda|^{1/\beta}}e^{t\Rep\lambda}\,dv(f,g^*,\lambda)
\\
\hfill
\text{as in \eqref{iinterm2};}
\\
\shoveleft{
\le \sup_{\{g^*\in X^*\,|\,\|g^*\|=1\}}
\int\limits_{\left\{\lambda\in\sigma(A)\cap{\mathscr P}_{b_-,b_+}^\beta\,\middle|\,\Rep\lambda\le -c(b_-),\, e^{s|\lambda|^{1/\beta}}e^{t\Rep\lambda}>n\right\}}
e^{\left[t-s2^{1/\beta}b_-^{-1}\right]\Rep\lambda}\,dv(f,g^*,\lambda)
}\\
\hfill
\text{by the choice of $b_->0$ (see \eqref{b_-});}
\\
\shoveleft{
=\sup_{\{g^*\in X^*\,|\,\|g^*\|=1\}}
\int\limits_{\left\{\lambda\in\sigma(A)\cap{\mathscr P}_{b_-,b_+}^\beta\,\middle|\,\Rep\lambda\le -c(b_-),\, e^{s|\lambda|^{1/\beta}}e^{t\Rep\lambda}>n\right\}}1\,dv(f,g^*,\lambda)
}\\
\hfill
\text{by \eqref{cond(ii)};}
\\
\shoveleft{
\le \sup_{\{g^*\in X^*\,|\,\|g^*\|=1\}}
4M\left\|E_A\left(\left\{\lambda\in\sigma(A)\cap{\mathscr P}_{b_-,b_+}^\beta\,\middle|\,\Rep\lambda\le -c(b_-),\, e^{s|\lambda|^{1/\beta}}e^{t\Rep\lambda}>n\right\}\right)
f\right\|\|g^*\|
}\\
\shoveleft{
\le 4M\left\|E_A\left(\left\{\lambda\in\sigma(A)\cap{\mathscr P}_{b_-,b_+}^\beta\,\middle|\,\Rep\lambda\le -c(b_-),\, e^{s|\lambda|^{1/\beta}}e^{t\Rep\lambda}>n\right\}\right)
f\right\|
}\\
\hfill
\text{by the strong continuity of the {\it s.m.};}
\\
\ \
\to 4M\left\|E_A\left(\emptyset\right)f\right\|=0,\ n\to\infty.
\hfill
\end{multline*}

By Proposition \ref{prop} and the properties of the \textit{o.c.} (see {\cite[Theorem XVIII.$2.11$ (f)]{Dun-SchIII}}), \eqref{ffirst} and \eqref{ssecond} jointly imply that, for any $t>0$ and $s>0$,
\[
f\in D(e^{s|A|^{1/\beta}}e^{tA}),
\]
which, in view of \eqref{GC}, further implies that, for each $t>0$, 
\begin{equation*}
y(t)=e^{tA}f\in \bigcap_{s>0} D(e^{s|A|^{1/\beta}})
={\mathscr E}^{(\beta)}(A).
\end{equation*}

Whence, by Proposition \ref{particular}, we infer that
\begin{equation*}
y(\cdot) \in {\mathscr E}^{(\beta)}((0,\infty),X),
\end{equation*}
which completes the proof of the \textit{``if"} part.

\medskip
\textit{``Only if"} part. Let us prove this part {\it by contrapositive} assuming that, for any $b_+>0$, there exists a $b_->0$ such that the set 
$\sigma(A)\setminus {\mathscr P}_{b_-,b_+}^\beta$ is \textit{unbounded}. 

Let us show that, under the circumstances, we can equivalently set the following seemingly stronger hypothesis: there exists a $b_->0$ such that, for any $b_+>0$, the set 
$\sigma(A)\setminus{\mathscr P}^\beta_{b_-,b_+}$ is unbounded.

Indeed, under the premise, there are two possibilities:
\begin{enumerate}
\item  For some $b_->0$, the set
\[
\left\{\lambda\in\sigma(A)\,\middle|\,
-b_-|\Imp\lambda|^{1/\beta}<\Rep\lambda\le 0\right\}
\] 
is {\it unbounded}.
\item  For any $b_->0$, the set 
\[
\left\{\lambda\in\sigma(A)\,\middle|\,
-b_-|\Imp\lambda|^{1/\beta}<\Rep\lambda\le 0\right\}
\]
is {\it bounded}.
\end{enumerate}

In the first case, as is easily seen, the set 
$\sigma(A)\setminus{\mathscr P}^\beta_{b_-,b_+}$ 
is also unbounded for some $b_->0$ and any $b_+>0$.

In the second case, by the premise, we infer that, for any $b_+>0$, unbounded is the set
\[
\left\{\lambda\in\sigma(A)\,\middle|
0<\Rep\lambda<b_+|\Imp\lambda|^{1/\beta}\,\right\},
\]
which makes the set $\sigma(A)\setminus{\mathscr P}_{b_-,b_+}^\beta$ unbounded for any $b_->0$ and $b_+>0$.
 
The foregoing equivalent version of the premise implies, in particular, that, for some $b_->0$ and any $n\in \N$, unbounded is the set
\begin{equation*}
\sigma(A)\setminus {\mathscr P}^\beta_{b_-,n^{-2}}=
\left\{\lambda \in \sigma(A)\,\middle| 
-b_-|\Imp\lambda|^{1/\beta}<\Rep\lambda < n^{-2}|\Imp\lambda|^{1/\beta}\right\}.
\end{equation*} 

Hence, we can choose a sequence  $\left\{\lambda_n\right\}_{n=1}^\infty$ 
of points in the complex plane as follows:
\begin{equation*}
\begin{split}
&\lambda_n \in \sigma(A),\ n\in \N,\\
&-b_-|\Imp\lambda_n|^{1/\beta}<\Rep\lambda_n <n^{-2}|\Imp\lambda_n|^{1/\beta},\ n\in \N,\\
&\lambda_0:=0,\ |\lambda_n|>\max\left[n,|\lambda_{n-1}|\right],\ n\in \N.\\
\end{split}
\end{equation*}

The latter implies, in particular, that the points $\lambda_n$, $n\in\N$, are \textit{distinct} ($\lambda_i \neq \lambda_j$, $i\neq j$).

Since, for each $n\in \N$, the set
\begin{equation*}
\left\{ \lambda \in {\mathbb C}\,\middle|\, 
-b_-|\Imp\lambda|^{1/\beta}<\Rep\lambda <n^{-2}|\Imp\lambda|^{1/\beta},\
|\lambda|>\max\bigl[n,|\lambda_{n-1}|\bigr]\right\}
\end{equation*}
is {\it open} in $\C$, along with the point $\lambda_n$, it contains an {\it open disk}
\begin{equation*}
\Delta_n:=\left\{\lambda \in \C\, \middle|\,|\lambda-\lambda_n|<\varepsilon_n \right\}
\end{equation*} 
centered at $\lambda_n$ of some radius $\varepsilon_n>0$, i.e., for each $\lambda \in \Delta_n$,
\begin{equation}\label{ddisks1}
-b_-|\Imp\lambda|^{1/\beta}<\Rep\lambda < n^{-2}|\Imp\lambda|^{1/\beta}\ \text{and}\ |\lambda|>\max\bigl[n,|\lambda_{n-1}|\bigr].
\end{equation}

Furthermore, under the circumstances, we can regard the radii of the disks to be small enough so that
\begin{equation}\label{rradii1}
\begin{split}
&0<\varepsilon_n<\dfrac{1}{n},\ n\in\N,\ \text{and}\\
&\Delta_i \cap \Delta_j=\emptyset,\ i\neq j
\quad \text{(i.e., the disks are {\it pairwise disjoint})}.
\end{split}
\end{equation}

Whence, by the properties of the {\it s.m.}, 
\begin{equation*}
E_A(\Delta_i)E_A(\Delta_j)=0,\ i\neq j,
\end{equation*}
where $0$ stands for the \textit{zero operator} on $X$.

Observe also, that the subspaces $E_A(\Delta_n)X$, $n\in \N$, are \textit{nontrivial} since
\[
\Delta_n \cap \sigma(A)\neq \emptyset,\ n\in\N,
\]
with $\Delta_n$ being an \textit{open set} in $\C$. 

By choosing a unit vector $e_n\in E_A(\Delta_n)X$ for each $n\in\N$, we obtain a sequence 
$\left\{e_n\right\}_{n=1}^\infty$ such that
\begin{equation}\label{oortho1}
\|e_n\|=1,\ n\in\N,\ \text{and}\ E_A(\Delta_i)e_j=\delta_{ij}e_j,\ i,j\in\N,
\end{equation}
where $\delta_{ij}$ is the \textit{Kronecker delta}.

As is easily seen, \eqref{oortho1} implies that the vectors $e_n$, $n\in \N$, are \textit{linearly independent}.

Furthermore, there exists an $\varepsilon>0$ such that
\begin{equation}\label{ddist1}
d_n:=\dist\left(e_n,\spa\left(\left\{e_i\,|\,i\in\N,\ i\neq n\right\}\right)\right)\ge\varepsilon,\ n\in\N.
\end{equation}

Indeed, the opposite implies the existence of a subsequence $\left\{d_{n(k)}\right\}_{k=1}^\infty$ such that
\begin{equation*}
d_{n(k)}\to 0,\ k\to\infty.
\end{equation*}

Then, by selecting a vector
\[
f_{n(k)}\in 
\spa\left(\left\{e_i\,|\,i\in\N,\ i\neq n(k)\right\}\right),\ k\in\N,
\] 
such that 
\[
\|e_{n(k)}-f_{n(k)}\|<d_{n(k)}+1/k,\ k\in\N,
\]
we arrive at
\begin{multline*}
1=\|e_{n(k)}\|
\hfill
\text{since, by \eqref{oortho1}, 
$E_A(\Delta_{n(k)})f_{n(k)}=0$;}
\\
\shoveleft{
=\|E_A(\Delta_{n(k)})(e_{n(k)}-f_{n(k)})\|\
\le \|E_A(\Delta_{n(k)})\|\|e_{n(k)}-f_{n(k)}\|
\hfill
\text{by \eqref{bounded};}
}\\
\ \
\le M\|e_{n(k)}-f_{n(k)}\|\le M\left[d_{n(k)}+1/k\right]
\to 0,\ k\to\infty,
\hfill
\end{multline*}
which is a \textit{contradiction} proving \eqref{ddist1}. 

As follows from the {\it Hahn-Banach Theorem}, for any $n\in\N$, there is an $e^*_n\in X^*$ such that 
\begin{equation}\label{HH-B1}
\|e_n^*\|=1,\ n\in\N,\ \text{and}\ \langle e_i,e_j^*\rangle=\delta_{ij}d_i,\ i,j\in\N.
\end{equation}

Let us consider separately the two possibilities concerning the sequence of the real parts $\{\Rep\lambda_n\}_{n=1}^\infty$: its being \textit{bounded} or \textit{unbounded}. 

The case of the sequence's $\{\Rep\lambda_n\}_{n=1}^\infty$ 
being \textit{bounded} is considered 
in absolutely the same manner as the corresponding case in the proof of the \textit{"only if"} part of Theorem \ref{open1} and furnishes a weak solution $y(\cdot)$ of equation \eqref{1} such that
\begin{equation*}
y(1)\not\in {\mathscr E}^{\{\beta\}}(A).
\end{equation*}

Hence, by Proposition \ref{particular}, $y(\cdot)$ does not belong to the Roumieu type Gevrey class ${\mathscr E}^{\{\beta\}}\left( (0,\infty),X\right)$, 
and the more so, the narrower Beurling type Gevrey class ${\mathscr E}^{(\beta)}\left( (0,\infty),X\right)$. 

Now, suppose that the sequence $\{\Rep\lambda_n\}_{n=1}^\infty$
is \textit{unbounded}. 

Therefore, there is a subsequence $\{\Rep\lambda_{n(k)}\}_{k=1}^\infty$ such that
\[
\Rep\lambda_{n(k)}\to \infty \ \text{or}\ \Rep\lambda_{n(k)}\to -\infty,\ k\to \infty.
\]

Let us consider separately each of the two cases.

The case of 
\[
\Rep\lambda_{n(k)}\to \infty,\ k\to \infty
\] 
is also considered in the same manner as the corresponding case in the proof of the \textit{"only if"} part of Theorem \ref{open1}, and again furnishes a weak solution $y(\cdot)$ of equation \eqref{1} such that
\begin{equation*}
y(1)\not\in {\mathscr E}^{\{\beta\}}(A).
\end{equation*}

Hence, by Proposition \ref{particular}, $y(\cdot)$ does not belong to the Roumieu type Gevrey class ${\mathscr E}^{\{\beta\}}\left( (0,\infty),X\right)$, let alone, the narrower Beurling type Gevrey class ${\mathscr E}^{(\beta)}\left( (0,\infty),X\right)$. 

Suppose that 
\[
\Rep\lambda_{n(k)}\to -\infty,\ k\to \infty.
\] 
Then, without loss of generality, we can regard that
\begin{equation}\label{iinfinity2}
\Rep\lambda_{n(k)} \le -k,\ k\in\N.
\end{equation}

Consider the element
\begin{equation*}
f:=\sum_{k=1}^\infty k^{-2}e_{n(k)}\in X,
\end{equation*}
which is well defined since $\{k^{-2}\}_{k=1}^\infty\in l_1$ and $\|e_{n(k)}\|=1$, $k\in\N$ (see \eqref{oortho1}).

By \eqref{oortho1},
\begin{equation}\label{ssubvectors2}
E_A(\cup_{k=1}^\infty\Delta_{n(k)})f=f\ \text{and}\
E_A(\Delta_{n(k)})f=k^{-2}e_{n(k)},\
k\in\N.
\end{equation}

For arbitrary $t\ge 0$ and any $g^*\in X^*$, 
\begin{multline}\label{ffirst3}
\int\limits_{\sigma(A)}e^{t\Rep\lambda}\,dv(f,g^*,\lambda)
\hfill
\text{by \eqref{decompose} as in \eqref{first1};}
\\
\shoveleft{
=\sum_{k=1}^\infty k^{-2}\int\limits_{\sigma(A)\cap\Delta_{n(k)}}e^{t\Rep\lambda}\,dv(e_{n(k)},g^*,\lambda)
}\\
\hfill
\text{since, for $\lambda\in \Delta_{n(k)}$, by
\eqref{iinfinity2} and \eqref{rradii1},}
\\
\hfill
\Rep\lambda=\Rep\lambda_{n(k)}+(\Rep\lambda-\Rep\lambda_{n(k)})\le \Rep\lambda_{n(k)}+|\lambda-\lambda_{n(k)}|\le -k+1\le 0;
\\
\shoveleft{
\le \sum_{k=1}^\infty k^{-2}\int\limits_{\sigma(A)\cap\Delta_{n(k)}}1\,dv(e_{n(k)},g^*,\lambda)
=\sum_{k=1}^\infty k^{-2}v(e_{n(k)},g^*,\Delta_{n(k)})
\hfill
\text{by \eqref{tv};}
}\\
\hspace{1.2cm}
\le \sum_{k=1}^\infty k^{-2}4M\|e_{n(k)}\|\|g^*\|
= 4M\|g^*\|\sum_{k=1}^\infty k^{-2}
<\infty.
\hfill
\end{multline}

Similarly, for arbitrary $t\ge 0$ and any $n\in\N$,
\begin{multline}\label{ssecond3}
\sup_{\{g^*\in X^*\,|\,\|g^*\|=1\}}
\int\limits_{\left\{\lambda\in\sigma(A)\,\middle|\,e^{t\Rep\lambda}>n\right\}}e^{t\Rep\lambda}\,dv(f,g^*,\lambda)
\hfill
\text{as in \eqref{ffirst3};}
\\
\shoveleft{
\le \sup_{\{g^*\in X^*\,|\,\|g^*\|=1\}}\sum_{k=1}^\infty k^{-2}
\int\limits_{\left\{\lambda\in\sigma(A)\,\middle|\,e^{t\Rep\lambda}>n\right\}\cap \Delta_{n(k)}}1\,dv(e_{n(k)},g^*,\lambda)
}\\
\hfill
\text{by \eqref{ssubvectors2};}
\\
\shoveleft{
=\sup_{\{g^*\in X^*\,|\,\|g^*\|=1\}}\sum_{k=1}^\infty 
\int\limits_{\left\{\lambda\in\sigma(A)\,\middle|\,e^{t\Rep\lambda}>n\right\}\cap \Delta_{n(k)}}1\,dv(E_A(\Delta_{n(k)})f,g^*,\lambda)
}\\
\hfill
\text{by \eqref{decompose};}
\\
\shoveleft{
=\sup_{\{g^*\in X^*\,|\,\|g^*\|=1\}}\int\limits_{\{\lambda\in\sigma(A)\,|\,e^{t\Rep\lambda}>n\}}1\,dv(E_A(\cup_{k=1}^\infty\Delta_{n(k)})f,g^*,\lambda)
\hfill
\text{by \eqref{ssubvectors2};}
}\\
\shoveleft{
=\sup_{\{g^*\in X^*\,|\,\|g^*\|=1\}}\int\limits_{\{\lambda\in\sigma(A)\,|\,e^{t\Rep\lambda}>n\}}1\,dv(f,g^*,\lambda)
\hfill
\text{by \eqref{cond(ii)};}
}\\
\shoveleft{
\le \sup_{\{g^*\in X^*\,|\,\|g^*\|=1\}}4M
\left\|E_A\left(\left\{\lambda\in\sigma(A)\,\middle|\,e^{t\Rep\lambda}>n\right\}\right)f\right\|\|g^*\|
}\\
\shoveleft{
\le 4M\|E_A(\{\lambda\in\sigma(A)\,|\,e^{t\Rep\lambda}>n\})f\|
}\\
\hfill
\text{by the strong continuity of the {\it s.m.};}
\\
\hspace{1.2cm}
\to 4M\left\|E_A\left(\emptyset\right)f\right\|=0,\ n\to\infty.
\hfill
\end{multline}

By Proposition \ref{prop}, \eqref{ffirst3} and \eqref{ssecond3} jointly imply that 
\[
f\in \bigcap\limits_{t\ge 0}D(e^{tA}),
\]
and hence, by Theorem \ref{GWS},
\[
y(t):=e^{tA}f,\ t\ge 0,
\]
is a weak solution of equation \eqref{1}.

Let
\begin{equation}\label{functional2}
h^*:=\sum_{k=1}^\infty k^{-2}e_{n(k)}^*\in X^*,
\end{equation}
the functional being well defined since $\{k^{-2}\}_{k=1}^\infty\in l_1$ and $\|e_{n(k)}^*\|=1$, $k\in\N$ (see \eqref{HH-B1}).

In view of \eqref{HH-B1} and \eqref{ddist1}, we have:
\begin{equation}\label{funct-dist2}
\langle e_{n(k)},h^*\rangle=\langle e_{n(k)},k^{-2}e_{n(k)}^*\rangle=d_{n(k)}k^{-2}\ge \varepsilon k^{-2},\ k\in\N.
\end{equation}

Since, for any $\lambda \in \Delta_{n(k)}$, $k\in \N$, by \eqref{iinfinity2} and \eqref{rradii1},
\begin{multline}\label{-iinf}
\Rep\lambda =\Rep\lambda_{n(k)}+(\Rep\lambda-\Rep\lambda_{n(k)})
\le
\Rep\lambda_{n(k)}+|\Rep\lambda-\Rep\lambda_{n(k)}|
\\
\hspace{1.2cm}
\le 
\Rep\lambda_{n(k)}+\varepsilon_{n(k)}\le -k+1 \le 0
\hfill
\end{multline}
and, by \eqref{ddisks1},
\[
-b_-|\Imp\lambda|^{1/\beta}<\Rep\lambda,
\]
we infer that, for any $\lambda \in \Delta_{n(k)}$, $k\in \N$,
\begin{equation*}
|\lambda|\ge|\Imp\lambda|\ge 
\left[b_-^{-1}(-\Rep\lambda)\right]^\beta.
\end{equation*}

Using this estimate, for 
\begin{equation}\label{ss}
s:=2b_->0
\end{equation}
and the functional $h^*\in X^*$ defined by \eqref{functional2}, we have:
\begin{multline}\label{nnotin3}
\int\limits_{\sigma(A)}e^{s|\lambda|^{1/\beta}}e^{\Rep\lambda}\,dv(f,h^*,\lambda)
\hfill
\text{by \eqref{decompose} as in \eqref{first1};}
\\
\shoveleft{
=\sum_{k=1}^\infty k^{-2}\int\limits_{\sigma(A)\cap\Delta_{n(k)}}e^{s|\lambda|^{1/\beta}}e^{\Rep\lambda}\,dv(e_{n(k)},h^*,\lambda)
}\\
\shoveleft{
\ge\sum_{k=1}^\infty k^{-2}\int\limits_{\Delta_{n(k)}}e^{[sb_-^{-1}-1](-\Rep \lambda)}\,dv(e_{n(k)},h^*,\lambda)
}\\
\hfill
\text{since $s:=2b_->0$ (see \eqref{ss});}
\\
\shoveleft{
=\sum_{k=1}^\infty k^{-2}\int\limits_{\sigma(A)\cap\Delta_{n(k)}}e^{-\Rep \lambda}\,dv(e_{n(k)},h^*,\lambda)
\hfill
\text{by \eqref{-iinf};}
}\\
\shoveleft{
\ge 
\sum_{k=1}^\infty k^{-2}e^{k-1}\int\limits_{\sigma(A)\cap\Delta_{n(k)}}1\,dv(e_{n(k)},h^*,\lambda)
=
\sum_{k=1}^\infty k^{-2}e^{k-1}v(e_{n(k)},h^*,\Delta_{n(k)})
}\\
\shoveleft{
\ge \sum_{k=1}^\infty k^{-2}e^{k-1}|
\langle E_A(\Delta_{n(k)})e_{n(k)},h^*\rangle|
\hfill
\text{by \eqref{oortho1} and \eqref{funct-dist2};}
}\\
\hspace{1.2cm}
\ge \sum_{k=1}^\infty \varepsilon
k^{-4}e^{k-1}
=\infty.
\hfill
\end{multline} 

By Proposition \ref{prop} and the properties of the \textit{o.c.} (see {\cite[Theorem XVIII.$2.11$ (f)]{Dun-SchIII}}), \eqref{nnotin3} implies that
\[
f\notin D(e^{s|A|^{1/\beta}}e^{A})
\]
with $s=2b_->0$, which, in view of \eqref{GC}, further implies that
\begin{equation*}
y(1)=e^{A}f\notin \bigcap_{s>0} D(e^{s|A|^{1/\beta}})
={\mathscr E}^{(\beta)}(A).
\end{equation*}

Whence, by Proposition \ref{particular}, we infer that the weak solution $y(t)=e^{tA}f$, $t\ge 0$, 
of equation \eqref{1} does not belong to the Beurling type Gevrey class ${\mathscr E}^{(\beta)}\left( (0,\infty),X\right)$, which completes our consideration of the case of
the sequence's $\{\Rep\lambda_n\}_{n=1}^\infty$ being \textit{unbounded}. 

With every possibility concerning $\{\Rep\lambda_n\}_{n=1}^\infty$ considered, 
the proof by contrapositive of the \textit{``only if" part} is complete and so is the proof of the 
entire statement.
\end{proof}

\begin{rem}
Thus, we have obtained a generalization of {\cite[Theorem $4.3$]{Markin2001(1)}}, the counterpart for a normal operator $A$ in a complex Hilbert space,
and of {\cite[Corollary $4.1$]{Markin2016}}, a characterization of the generation of a Berling type Gevrey ultradifferentiable $C_0$-semigroup by a scalar type spectral operator $A$.
\end{rem}

\section{Inherent Smoothness Improvement Effect}

Now, let us see that there is more to be said about the important particular case of \textit{analyticity} ($\beta=1$) in Theorem \ref{open1} (see Corollary \ref{CAWS}).

\begin{prop}\label{sector}
Let $A$ be a scalar type spectral operator in a complex Banach space $(X,\|\cdot\|)$. If every weak solution of equation \eqref{1} is analytically continuable into a complex neighborhood of $(0,\infty)$ (each one into its own), then all of them are analytically continuable into the open sector
\begin{equation*}
\Sigma_\theta:=\left\{\lambda\in \C\,\middle|\,|\arg \lambda|<\theta \right\}\setminus \{0\}
\end{equation*}
with 
\[
\theta:=\sup 
\left\{0<\varphi<\pi/2\,\middle|\, 
\left\{\lambda\in \sigma(A)\,\middle|\,\Rep\lambda <0,\
|\arg \lambda|\le\pi/2+\varphi \right\}\
\text{is bounded}\right\},
\]
where $-\pi<\arg\lambda \le \pi$ is the principal value of the argument of $\lambda$ \textup{($\arg 0:=0$)}.
\end{prop}

\begin{proof}
By Corollary \ref{CAWS}, the analyticity of all weak solutions of equation \eqref{1} on $(0,\infty)$ is equivalent to the existence of $b_+>0$ and $b_->0$ such that the set
\begin{equation*}
\sigma(A)\setminus \left\{\lambda\in \C\,\middle|\,
\Rep\le -b_-|\Imp \lambda|\ \text{or}\ \Rep\ge b_+|\Imp \lambda| \right\}
\end{equation*}
is \textit{bounded} (see Fig. \ref{fig:graph6}).

As is easily seen, this implies, in particular, that the set
\[
\Phi:=\left\{0<\varphi<\pi/2\,\middle|\, 
\left\{\lambda\in \sigma(A)\,\middle|\,\Rep\lambda <0,\
|\arg \lambda|\le\pi/2+\varphi \right\}\
\text{is bounded}\right\}\neq \emptyset.
\]

For any $\varphi \in \Phi$, 
\[
A=A_\varphi^-+A_\varphi^+,
\]
where the scalar type spectral operators
$A_\varphi^-$ and $A_\varphi^+$ are defined as follows:
\begin{equation*}
\begin{split}
A_\varphi^-&:=AE_A\left(\left\{\lambda \in \sigma(A)\,
\middle|\,|\arg \lambda|\ge \pi/2+\varphi\right\}\right),
\\
A_\varphi^+&:=AE_A\left(\left\{\lambda \in \sigma(A)\,
\middle|\,|\arg \lambda|<\pi/2+\varphi\right\}\right) 
\end{split}
\end{equation*}
(see {\cite[Theorem XVIII.$2.11$ (f)]{Dun-SchIII}}).

By the properties of the {\it o.c.} (see {\cite[Theorem XVIII.$2.11$ (h), (c)]{Dun-SchIII}}),
for any $\varphi \in \Phi$,
\begin{equation*}
\begin{split}
\sigma(A_\varphi^-)&\subseteq \left\{\lambda \in \sigma(A)\,
\middle|\,|\arg \lambda|\ge \pi/2+\varphi\right\}
\cup \{0\},
\\
\sigma(A_\varphi^+)&\subseteq 
\left\{\lambda \in \sigma(A)\,
\middle|\,|\arg \lambda|\le \pi/2+\varphi\right\}. 
\end{split}
\end{equation*}

Hence, by {\cite[Proposition $4.1$]{Markin2002(2)}}
(cf. also \cite{Markin2004(1)}), for any $\varphi \in \Phi$, the operator $A_\varphi^-$ generates the $C_0$-semigroup
$\left\{e^{tA_\varphi^-}\right\}_{t\ge 0}$ 
of the operator exponentials (see Preliminaries) {\it analytic} in the sector
\begin{equation*}
\Sigma_\varphi:=\left\{\lambda\in \C\,\middle|\,\arg \lambda <\varphi \right\}\setminus \{0\}
\end{equation*}
(see also \cite{Engel-Nagel}).

As follows from the premise, for any $\varphi \in \Phi$,
the set 
\begin{equation*}
\sigma(A_\varphi^+)\setminus \left\{ \lambda \in \C\,\middle|\,\Rep\lambda \ge b_+|\Imp\lambda|\right\},
\end{equation*}
is \textit{bounded}, which, by {\cite[Corollary $4.1$]{Markin2018(4)}}, implies that all weak solutions of the equation
\begin{equation*}
y'(t)=A_\varphi^+y(t),\ t\ge 0,
\end{equation*}
i.e., by Theorem \ref{GWS}, all vector functions of the form
\[
y(t)=e^{tA_\varphi^+}f,\ t\ge 0,f\in \bigcap_{t\ge 0}D(e^{tA_\varphi^+}) 
\]
are \textit{entire}.

By the properties of the {\it o.c.} (see {\cite[Theorem XVIII.$2.11$]{Dun-SchIII}}),
\begin{equation*}
e^{tA}=e^{tA_\varphi^-}+e^{tA_\varphi^+}-I,\ t\ge 0.
\end{equation*}

In view of the fact that
\[
D(e^{tA_\varphi^-})=X,\ t\ge 0,
\]
for each
\[
f\in \bigcap_{t\ge 0}D(e^{tA})=\bigcap_{t\ge 0}D(e^{tA_\varphi^+}),
\]
the vector function 
\[
y_+(t):=\left[e^{tA_\varphi^+}-I\right]f,\ t\ge 0,
\]
is entire, whereas the vector function
\[
y_-(t):=e^{tA_\varphi^-}f,\ t\ge 0,
\]
is analytically continuable into the open sector $\Sigma_\varphi$, which makes the
the vector function 
\[
y(t):=e^{tA}f=y_-(t)+y_+(t),\ t\ge 0,
\]
to be analytically continuable into the open sector $\Sigma_\varphi$.

Considering  that 
\[
\varphi\in\Phi\ \text{and}\ f\in \bigcap_{t\ge 0}D(e^{tA})
\]
are arbitrary, by Theorem \ref{GWS}, we infer that every weak solution of equation \eqref{1} is analytically continuable into the sector
\begin{equation*}
\Sigma_\theta:=\left\{\lambda\in \C\,\middle|\,|\arg \lambda|<\theta \right\}\setminus \{0\}
\end{equation*}
with $\theta:=\sup \Phi$.
\end{proof}

\begin{rems}\
\begin{itemize}
\item Thus, we have obtained a generalization of 
{\cite[Proposition $5.2$]{Markin2001(1)}}, the counterpart for a normal operator $A$ in a complex Hilbert space.
\item It is noteworthy that Corollary \ref{CAWS}
(i.e., Theorem \ref{open1} with $\beta=1$) and Proposition \ref{sector} with $\theta=\pi/2$ apply to equation \eqref{1} with a \textit{self-adjoint operator} in a complex Hilbert space, which implies that, for such an equation, all weak solutions are analytically continuable into the open right half-plane
\[
\left\{\lambda\in \C\,\middle|\,\Rep\lambda>0 \right\}
\]
(see {\cite[Corollary $5.1$]{Markin2001(1)}} and, for symmetric operators, {\cite[Theorem $6.1$]{Markin2001(1)}}).
\end{itemize} 
\end{rems} 

\section{Concluding Remark}

Due to the {\it scalar type spectrality} of the operator $A$, Theorems \ref{open1} and \ref{open2} are stated exclusively in terms of the location of its {\it spectrum} in the complex plane, similarly to the celebrated \textit{Lyapunov stability theorem} \cite{Lyapunov1892} 
(cf. {\cite[Ch. I, Theorem 2.10]{Engel-Nagel}}), and thus, are intrinsically qualitative statements (cf. \cite{Markin2011,Markin2018(4),Yosida1958}).

\section{Acknowledgments}

The author extends sincere gratitude to his colleague, Dr.~Maria Nogin of the Department of Mathematics, California State University, Fresno, for her kind assistance with the graphics.




\end{document}